\documentclass[12pt]{amsart}
\usepackage{amsmath,amssymb,amsthm,tikz,stmaryrd}
\usepackage[letterpaper,margin=1in]{geometry}

\newtheorem{theorem}{\textbf{Theorem}}[section]
\newtheorem*{theorem*}{\textbf{Theorem}}
\newtheorem*{conjecture*}{\textbf{Conjecture}}
\newtheorem{proposition}[theorem]{\textbf{Proposition}}
\newtheorem{desideratum}[theorem]{\textbf{Desideratum}}
\newtheorem{corollary}[theorem]{\textbf{Corollary}}
\newtheorem{lemma}[theorem]{\textbf{Lemma}}

\newtheoremstyle{plainnonitalic} 
{}                
{}                
{\normalfont}     
{}                
{\bfseries}       
{.}               
{5pt plus 1pt minus 1pt} 
{}                

\theoremstyle{plainnonitalic}
\newtheorem{example}[theorem]{Example}
\newtheorem{axiom}{Axiom}
\newtheorem{remark}{Remark}
\newtheorem{definition}{Definition}

\newcommand{\too}{to [out=0,in=180]}

\begin{document}

\title{The Six-Vertex Yang-Baxter Groupoid}

\author{Daniel Bump}
\address[D.~Bump]{Department of Mathematics, Stanford University, Stanford, CA 94305-2125}
\email{bump@math.stanford.edu}
\urladdr{https://math.stanford.edu/~bump/}

\author{Slava Naprienko}
\address{(S. Naprienko)} 
\email{slava@naprienko.com}


\begin{abstract}
  A parametrized Yang-Baxter equation is usually defined to be a map from a
  \textit{group} to a set of R-matrices, satisfying the Yang-Baxter commutation
  relation. These are a mainstay of solvable lattice models.
  We will show how the parameter space can sometimes be enlarged to
  a \textit{groupoid}, and give two examples of such groupoid parametrized
  Yang-Baxter equations, within the six vertex model. A \textit{groupoid
  parametrized Yang-Baxter equation} consists of a groupoid $\mathfrak{G}$
  together with a map $\pi:\mathfrak{G}\to\operatorname{End}(V\otimes V)$
  for some vector space $V$ such that the Yang-Baxter commutator
  $\llbracket \pi(u),\pi(w),\pi(v)\rrbracket=0$
  if $u,v\in\mathfrak{G}$ are such that the groupoid composition $w=u\star v$ is
  defined. An important role is played by an \textit{object map} $\Delta:\mathfrak{G}\to M$
  for some set $M$ such that $\Delta(u)=\Delta(v')$, $\Delta(w)=\Delta(v)$
  and $\Delta(w')=\Delta(u')$, where $v\mapsto v'$ is the groupoid inverse map.

  There are two main regimes of the six-vertex model: the free-fermionic point,
  and everything else. For the free-fermionic point, there exists a
  parametrized Yang-Baxter equation with a large parameter group
  $\operatorname{GL}(2)\times\operatorname{GL}(1)$.  For non-free-fermionic
  six-vertex matrices, there are also well-known (group) parametrized Yang-Baxter
  equations, but these do not account for all possible interactions.  Instead
  we will construct a \textit{groupoid} parametrized Yang-Baxter equation that
  accounts for essentially all possible Yang-Baxter equations in the six-vertex
  model. We will also exhibit a separate groupoid for the five-vertex model. 

	We will show how to construct solvable lattice models based on groupoid
  parametrized Yang-Baxter equations. This includes new systems that cannot be
	constructed using group parametrized Yang-Baxter equation.
	The fact that this is possible depends on the associativity axiom for a
	groupoid and illustrates the utility of groupoid parametrized Yang-Baxter
	equations.
  \end{abstract}

\maketitle

\section{Introduction}

Let $V$ be a finite-dimensional vector space. If $r, s, t \in \mathrm{End} (V
\otimes V)$, the \textit{Yang-Baxter commutator} is defined to be the
endomorphism of $V \otimes V \otimes V$ defined by
\[ \llbracket r, s, t \rrbracket = (r \otimes I_V)  (I_V \otimes s)  (t
   \otimes I_V) - (I_V \otimes t)  (s \otimes I_V)  (I_V \otimes r) . \]
Let $G$ be a group. A (group) \textit{parametrized Yang-Baxter equation}
with parameter group $G$ is a map $\pi : G \longrightarrow \mathrm{End} (V
\otimes V)$ such that
\begin{equation}
  \label{eq:paramybe} \llbracket \pi (g), \pi (gh), \pi (g) \rrbracket = 0.
\end{equation}
Group parametrized Yang-Baxter equations are a mainstay of solvable lattice
model theory, underlying much work starting with Baxter's treatment of the
field free six- and eight-vertex model~{\cite{BaxterBook}}. In this paper
we propose a generalization, replacing the group $G$ by a groupoid.
We review groupoids in Section~\ref{sec:groupoids}. 

To define the notion of a \textit{groupoid parametrized Yang-Baxter equation},
let $\mathfrak{G}$ be a groupoid, and $\pi : \mathfrak{G}
\longrightarrow \operatorname{End} (V \otimes V)$ a map such that if $u, v \in G$
and if the product $u \star v$ is defined in $\mathfrak{G}$, then $\llbracket
\pi (u), \pi (u \star v), \pi (v) \rrbracket = 0$. We will further assume that
there exists a set $M$ and a map $\Delta : \mathfrak{G} \longrightarrow M$
such that $u \star v$ is defined if and only if $\Delta (u) = \Delta (v')$,
where $v'$ is the groupoid inverse of $v$, and if $w = u \star v$ then
\begin{equation}
  \label{eq:deltatr} \Delta (u) = \Delta (v'), \qquad \Delta (w) = \Delta (v),
  \qquad \Delta (w') = \Delta (u') .
\end{equation}
If $\mathfrak{G}$ is a group, we may take $M$ to be a single point, so the map
$\Delta$ contains no information. We will call the function $\Delta :
\mathfrak{G} \longrightarrow M$ the {\textit{object map}}.

\begin{remark}
  The existence of the map $\Delta$ follows from the first part of the
  definition. Indeed, by Proposition~\ref{prop:gpoidcat} every groupoid is 
  associated with a groupoid category $\mathcal{C}$. Then $M$ can be taken to be
  the set of objects of $\mathcal{C}$, while the groupoid $\mathfrak{G}$ is
  the set of all morphisms in $\mathcal{C}$. The map $\Delta : \mathfrak{G}
  \longrightarrow M$ then associates with a morphism $v$ its source
  object~$\Delta (v)$. The target of the morphism is $\Delta (v')$. The
  groupoid composition law is the composition in the category $\mathcal{C}$,
  and all three identities in (\ref{eq:deltatr}) then have interpretations.
  The first identity says that if $u$ and $v$ are morphisms the composition $u
  \circ v$ is defined if and only if the source of $u$ is the target of $v$,
  and the other identities say that the source of $u \circ v$ is the source of
  $v$, and that the target of $u \circ v$ is the target of $u$.
  Although the existence of the map $\Delta$ is thus automatic, in practice
  $M$ is a concrete set and the map $\Delta$ plays an important role in the
  construction of the groupoid. Therefore we are making it a part of the definition.
\end{remark}

In the simplest case the groupoid is just a disjoint union of groups. In
this case, $M$ could be any set that is in bijection with the set of groups in
the disjoint union. If this is so, then $\Delta (v) =
\Delta (v')$, and $\llbracket u, w, v \rrbracket = 0$ implies that $\Delta (u)
= \Delta (v) = \Delta (w)$.

As proof of concept, in this paper we will give two groupoid parametrized
Yang-Baxter equations, related to the (non-free-fermionic) six-vertex model,
and the five-vertex models. These groupoids are \textit{not} disjoint unions
of groups.

Both of these groupoids contain well-known subgroups that are the basis of interesting
Yang-Baxter equations with applications to solvable lattice models. So to
demonstrate that groupoid parametrized are potentially important, in
Section~\ref{sec:solvable} we will consider some solvable lattice models based
on an arbitrary groupoid parametrized Yang-Baxter equations. We will see that
such lattice models can be constructed with both row and column solvability.
These are similar to well-known models (for example~{\cite{BaxterBook,KuperbergASM,BBF11}})
which indeed correspond to subgroups of the groupoid. But if we work with a
groupoid parametrized Yang-Baxter equation in which the groupoid is not a
disjoint union of groups, the examples in Section~\ref{sec:solvable} are more
general. Also Zhong~\cite{ZhongSymplectic} contains an example of stochastic
Yang-Baxter equations that can profitably be analyzed from the groupoid
point of view.

Let us mention three questions about this work.
First, we believe that the groupoid parametrized Yang-Baxter equations in
this paper are not isolated example. In fact, it may be seen that there
are interesting groupoid parametrized Yang-Baxter equations within the
(non-field free) eight vertex model. It is possible to believe that
Yang-Baxter groupoids are abundant, though perhaps not easy to find.

The second question is whether the six-vertex groupoid parametrized
Yang-Baxter equation in this paper has a more uniform construction.
Our construction of the groupoid in this paper depends on delicate considerations regarding the boundary
components of the groupoid. Possibly the groupoid can be constructed
algebraically as an action groupoid, or by other tools such as those
developed in~\cite{MackenzieGroupoids}. With the groupoid 
constructed algebraically some difficulties could perhaps be avoided
in the Yang-Baxter equation. 

The third question is how these Yang-Baxter groupoids fit into the
theoretical framework of quantum groups or Hopf algebroids.

\subsection{The Six-Vertex Model}

Let $V =\mathbb{C}^2$, with standard basis $\mathbf{e}_1, \mathbf{e}_2$. A
\textit{six-vertex matrix} is an endomorphism of $V \otimes V$, with basis
$\mathbf{e}_1 \otimes \mathbf{e}_1, \mathbf{e}_1 \otimes \mathbf{e}_2$,
$\mathbf{e}_2 \otimes \mathbf{e}_1$, $\mathbf{e}_2 \otimes \mathbf{e}_2$,
having a matrix of the form
\begin{equation}
  \label{eq:sixvertex} u = \left( \begin{array}{cccc}
    a_1 (u) &  &  & \\
    & c_1 (u) & b_1 (u) & \\
    & b_2 (u) & c_2 (u) & \\
    &  &  & a_2 (u)
  \end{array} \right) .
\end{equation}
Furthermore we require that $c_1 (u), c_2 (u)$ is nonzero. The entries in the
matrix (\ref{eq:sixvertex}) are called \textit{Boltzmann weights} due to the
origin of the subject in statistical mechanics~{\cite{BaxterBook}}.

Let us give two examples of group parametrized Yang-Baxter equations in the
six-vertex model. First, let $q_1$ and $q_2$ be fixed, nonzero complex
numbers. For $z_1, z_2, w \in \mathbb{C}^{\times}$, define a six-vertex matrix
$R^{\mathrm{cf}}_{q_1, q_2} (z_1, z_2, w)$ to be the matrix with the following
Boltzmann weights:
\[ \begin{array}{|l|l|l|l|l|l|}
     \hline
     a_1 (R) & a_2 (R) & b_1 (R) & b_2 (R) & c_1 (R) & c_2 (R)\\
     \hline
     q_1 z_1 - q_2 z_2 & q_1 z_1 - q_2 z_2 & q_1  (z_1 - z_2) & q_2  (z_1 -
     z_2) & z_1 w (q_1 - q_2) & z_2 w^{- 1}  (q_1 - q_2)\\
     \hline
   \end{array} \]
Then one may check that
\[ \llbracket R^{\mathrm{cf}}_{q_1, q_2} (z_1, z_2, w_1),
   R^{\mathrm{cf}}_{q_1, q_2} (z_1 z_3, z_2 z_4, w_1 w_2),
   R^{\mathrm{cf}}_{q_1, q_2} (z_3, z_4, w_2) \rrbracket = 0. \]
So this is an example of a parametrized Yang-Baxter equation in the six-vertex
model with parameter group $(\mathbb{C}^{\times})^3$.

A second example is, with $z_1, z_2, w \in \mathbb{C}^{\times}$, the
$R^{\mathrm{ff}}_{q_1, q_2} (z_1, z_2, w)$ to be the matrix with the following
Boltzmann weights:
\[ \begin{array}{|l|l|l|l|l|l|}
     \hline
     a_1 (R) & a_2 (R) & b_1 (R) & b_2 (R) & c_1 (R) & c_2 (R)\\
     \hline
     q_1 z_1 - q_2 z_2 & q_1 z_2 - q_2 z_1 & q_1  (z_1 - z_2) & q_2  (z_1 -
     z_2) & z_1 w (q_1 - q_2) & z_2 w^{- 1}  (q_1 - q_2)\\
     \hline
   \end{array} \]
Again, we have
\[ \llbracket R^{\mathrm{ff}}_{q_1, q_2} (z_1, z_2, w_1),
   R^{\mathrm{ff}}_{q_1, q_2} (z_1 z_3, z_2 z_4, w_1 w_2),
   R^{\mathrm{ff}}_{q_1, q_2} (z_3, z_4, w_2) \rrbracket = 0. \]
So this is also a parametrized Yang-Baxter equation in the six-vertex model
with parameter group $(\mathbb{C}^{\times})^3$.

These two examples are very similar: the weights are the same
\textit{except} for that $a_2$ entries, which differ. The weights
$R^{\mathrm{ff}}_{q_1, q_2} (z_1, z_2, w)$ are \textit{free-fermionic},
meaning that they satisfy the condition
\[ a_1 (R) a_2 (R) + b_1 (R) b_2 (R) - c_1 (R) c_2 (R) = 0. \]
On the other hand the weights $R^{\mathrm{cf}}_{q_1, q_2} (z_1, z_2, w)$ are
not free fermionic, but instead satisfy the condition $a_1 (R) = a_2 (R)$.

Despite the similarity between the above two examples, there are nevertheless
important differences between the free-fermionic regime and the
non-free-fermionic regime in the six-vertex model. The free-fermionic example
$R^{\mathrm{ff}}_{q_1, q_2}$ can be extended to a parametrized Yang-Baxter
equation with a larger group $\mathrm{GL} (2, \mathbb{C}) \times \mathrm{GL}
(1, \mathbb{C})$. In contrast $R^{\mathrm{cf}}_{q_1, q_2}$ does not extend to
a properly larger connected group.

The dichotomy between the free-fermionic and non-free-fermionic cases can be
explained as follows. Drinfeld~{\cite{DrinfeldQuantumGroups}} and
Jimbo~{\cite{Jimbo}} showed that solutions to the Yang-Baxter equation come
from quantum groups. For non-free-fermionic parametrized Yang-Baxter equations
such as $R^{\mathrm{cf}}$, the relevant quantum group is $U_q 
(\mathfrak{g}\mathfrak{l}_2)$ or its affinization, or Drinfeld twists,
including two-parameter quantum groups. For free-fermionic Yang-Baxter
equation such as $R^{\mathrm{ff}}$, the relevant quantum group is the
superalgebra $U_q  (\mathfrak{s}\mathfrak{l}(1|1))$, or its affinization
({\cite{BazhanovShadrikov,BrackenGouldZhangDelius,BBGSuper}}).

Differences between the representation theories of $U_q (\mathfrak{gl}_2)$ and
$U_q (\mathfrak{sl} (1|1))$ explain some of the
differences between these cases. Thus $U_q (\mathfrak{sl} (1|1))$ has many
irreducible two-dimensional representations, accounting for the large
$\mathrm{GL} (2) \times \mathrm{GL} (1)$-parametrized Yang-Baxter equation in
the free-fermionic case. In contrast, this is not true for $U_q
(\mathfrak{sl}_2)$, which has essentially only one. (Its affinization has a
one-parameter family of two-dimensional representations.)

The group parametrized Yang-Baxter equations in the non-free-fermionic case do
not account for all possible interactions between six-vertex matrices. To
demonstrate this, let us consider the parametrized Yang-Baxter equation
$R^{\mathrm{cf}}_{q_1, q_2}$. We choose an element $r = R^{\mathrm{cf}}_{q_1,
q_2} (z_1, z_2, w)$ of this group. There are always six-vertex matrices $u$
and $v$ that are outside the group such that $\llbracket u, v, r \rrbracket =
0$. Indeed, by Corollary~\ref{cor:conddel} below, if $u$ is any six-vertex
matrix such that $\Delta (u) = \Delta (r^{\ast}) = (\frac{q_1 + q_2}{q_2},
\frac{q_1 + q_2}{q_1})$, then there exists a $u$ such that $\llbracket u, v, r
\rrbracket = 0$. The condition on $u$ restricts it to a four-dimensional
variety, but the group of matrices $R^{\mathrm{cf}}_{q_1, q_2}$ is
three-dimensional, so $u$ and $v$ that are outside the group certainly exist.

\begin{theorem*}
  There is a groupoid parametrized Yang-Baxter equation $\pi : \mathfrak{G}
  \to \mathrm{End} (V \otimes V)$ with $V =\mathbb{C}^2$ with object map
  $\Delta : \mathfrak{G} \longrightarrow (\mathbb{C}^{\times})^2$ such that if
  $u$ and $v$ are six-vertex matrices in general position then there exists a
  six-vertex matrix $w$ such that $\llbracket u, w, v \rrbracket = 0$ if and
  only if there exist $g, h \in \mathfrak{G}$ such that $u = \pi (g)$ and $v =
  \pi (h)$, and $g \star h$ is defined. Then $w = \pi (g \star h)$ is the
  solution.
\end{theorem*}

Thus this groupoid parametrized Yang-Baxter equation accounts for all
Yang-Baxter identities $\llbracket r, s, t \rrbracket = 0$ with $r, s, t$ in
an open subset $S^{\circ}$ of six-vertex matrices. A disjoint union of
groupoids is a groupoid. The groupoid $\mathfrak{G}$ decomposes into the
disjoint union of a free-fermionic part $\mathfrak{G}_{\mathrm{ff}}$, which is
actually the group $\mathrm{GL} (2) \times \mathrm{GL} (1)$, and a more
interesting non-free-fermionic part $\mathfrak{G}_{\mathrm{nf}}$, which is a
groupoid, but not a disjoint union of groups.

\begin{remark}
  For $r, s, t$ outside $S^{\circ}$, there may be further boundary groupoids
  that we do not fully analyze. Particularly, there are \textit{five-vertex}
  matrices which have themselves important applications. We will 
  prove in Section~\ref{sec:fivever} that the matrices with $b_2 = 0$ can be
  organized into an interesting groupoid. The matrices with $b_1 = 0$
  similarly fit into a groupoid. These lower-dimensional groupoids are not
  included in $\mathfrak{G}$. We are aware of at least one example of a
	groupoid parametrized Yang-Baxter equation outside the six-vertex model.
	Hence we hope that groupoid parametrized Yang-Baxter equations will be
	an important phenomenon.
\end{remark}

This groupoid parametrized Yang-Baxter equation has the stronger property that
$g \star h$ is defined if and only if $\llbracket \pi (g), w, \pi (h)
\rrbracket = 0$ has a solution, and if this is the case then $w$ is a constant
multiple of $\pi (g \star h)$. This strong property justifies our assertion
that the groupoid parametrized Yang-Baxter equation accounts for essentially 
all interactions in the six-vertex model.

In this introductory section we will describe the groupoid $\mathfrak{G}$ in
rough terms, as an outline of the somewhat technical construction that will
follow in later sections. The map $\pi : \mathfrak{G}_{\mathrm{nf}} \to
\mathrm{End} (V \otimes V)$ is not an isomorphism onto its image. The set $\pi
(\mathfrak{G}_{\mathrm{nf}})$ is a specific set $\bar{\Omega}$ of matrices in
$\mathrm{End} (V \otimes V)$. The set $\bar{\Omega}$ can be thought of as a
quasi-affine variety, and it has a Zariski topology. The set $\pi :
\mathfrak{G}_{\mathrm{nf}} \rightarrow \bar{\Omega}$ is a morphism that is a
birational equivalence. Essentially, $\mathfrak{G}_{\mathrm{nf}}$ is obtained
by ``blowing up'' certain lower-dimensional boundary components.

The object map $\Delta$ takes values in $M = (\mathbb{C}^{\times})^2$. If $u
\in \Omega$ define
\[ N (u) = a_1 (u) a_2 (u) + b_1 (u) b_2 (u) - c_1 (u) c_2 (u) \]
\[ \Delta_1 (u) = \frac{N (u)}{a_1 (u) b_1 (u)}, \qquad \Delta_2 (u) = \frac{N
   (u)}{a_2 (u) b_2 (u)} . \]
We find that
\[ \Delta_1 (u') = \frac{N (u)}{a_1 (u) b_1 (u)}, \qquad \Delta_2 (u') =
   \frac{N (u)}{a_2 (u) b_2 (u)} . \]
Define $\Delta (u) = (\Delta_1 (u), \Delta_2 (u))$. This function is
holomorphic on $\Omega^{\circ}$ but it cannot be extended to all of
$\overline{\Omega}$ since the numerator and denominator can both vanish. This
is mitigated by blowing up two boundary components, and $\Delta$ extends to a
continuous function on all of $\mathfrak{G}$.

Although the six-vertex groupoid is not a disjoint union of groups, it can
still be decomposed into a disjoint union of smaller groupoids that we will
call {\textit{blocks}}.

\begin{remark}
  Define $\Delta_0 (u) = \Delta_1 (u) \Delta_2 (u)$. It may be checked that
  $\Delta_0 (u) = \Delta_0 (u')$. If $\delta_0 \in \mathbb{C}^{\times}$ is
	given, define the \textit{block}
  \[ \mathfrak{G}_{\mathrm{nf}} (d_0) = \{u \in \mathfrak{G}_{\mathrm{nf}} |
     \Delta_0 (u) = \delta_0 \} . \]
  Note that if $u$ and $v$ lie in different blocks then $\Delta (u) = \Delta
  (v')$ is impossible since this would imply that $\Delta_0 (u) = \Delta_0
  (v')$. Thus if $u, v, w \in \mathfrak{G}$ and $\llbracket u, w, v \rrbracket
  = 0$ all three $u, v, w$ must lie in the same block. Thus we have a further
  decomposition
  \[ \mathfrak{G}_{\mathrm{nf}} = \bigsqcup_{d_0} \mathfrak{G}_{\mathrm{nf}}
     (d_0) . \]
\end{remark}

\

Another paper relating parametrized Yang-Baxter equations to groupoids is
Felder and Ren~{\cite{FelderRen}}. However their use of a groupoid is
different from ours.

This paper is based on~{\cite{NaprienkoGroupoid}}, to which it adds analysis
of the boundary components of the Six-Vertex Groupoid.

We thank Amol Aggarwal, Ben Brubaker, Valentin Buciumas, Henrik Gustafsson,
Andrew Hardt and Travis Scrimshaw for helpful conversations and comments.

\section{Groupoids}\label{sec:groupoids}

A \textit{groupoid} is a set $G$ with a partially defined composition. This
consists of a map $\mu : S \longrightarrow G$, where $S$ is a subset of $G
\times G$. If $a, b \in G$ we say that the product $a \star b$ is
\textit{defined} if $(a, b) \in S$, and then we write $a \star b = \mu (a,
b)$. The groupoid is also required to have an ``inverse map'' $x \mapsto x'$
from $G \rightarrow G$. The inverse map is more commonly denoted as $x \mapsto
x^{- 1}$, but we will be concerned with a groupoid whose elements are
matrices, and we will reserve the notation $x^{- 1}$ for the matrix inverse.
The following axioms are required.

\begin{axiom}
  [Associative Law] If $a \star b$ and $b \star c$ are defined then $(a \star
  b) \star c$ and $a \star (b \star c)$ are defined, and they are equal.
\end{axiom}

We say that $a \star b \star c$ is \textit{defined} if $a \star b$ and $b
\star c$ are defined, and then we denote $(a \star b) \star c = a \star (b
\star c)$ as $a \star b \star c$.

\begin{axiom}
  [Inverse] The compositions $a \star a'$ and $a' \star a$ are always defined.
  Thus if $a \star b$ is defined, then $a \star b \star b'$ is defined, and
  this is required to equal $a$. Similarly $a' \star a \star b$ is defined,
  and this is required to equal~$b$.
\end{axiom}

\begin{example}
  \label{example:categorygroupoid}A category $\mathcal{C}$ is \textit{small}
  if its class of objects is a set. A small category is a \textit{groupoid
  category} if every morphism is an isomorphism. Assuming this, the disjoint
  union
  \[ G = \bigsqcup_{A, B \in \mathcal{C}} \mathrm{Hom} (A, B) \]
  is a groupoid, with the $\star$ operation being composition: thus if $a \in
  \mathrm{Hom} (A, B)$ and $b \in \mathrm{Hom} (C, D)$, then $a \star b$ is
  defined if and only if $B = C$. The groupoid axioms are clear.
\end{example}

\begin{lemma}
  \label{lem:gpoinvlem}In a groupoid, we have $(a')' = a$. Moreover if $a
  \star b$ is defined then so is $b' \star a'$ and $(a \star b)' = b' \star
  a'$.
\end{lemma}

\begin{proof}
  Since $(a')' \star a'$ and $a' \star a$ are both defined, by the Associative
  Law the product $(a')' \star a' \star a$ is defined, and using the Inverse
  Axiom, this equals both $(a')'$ and $a$. For the second assertion, assume $a
  \star b$ is defined. It follows from the axioms that
  \[ (a \star b)' = (a \star b)' \star a \star b \star b' \star a' = 
  b' \star a' . \qedhere \]

\end{proof}

Given a groupoid $G$, let us say an element $A$ is \textit{idempotent} if $A
\star A$ is defined and $A \star A = A$.

\begin{lemma}
  An element $A \in G$ is an idempotent if and only if $A = g \star g'$ for
  some $g \in G$. If $A$ is idempotent then $A = A'$.
\end{lemma}

\begin{proof}
  It is easy to check that $g \star g'$ is idempotent. Conversely if $A$ is
  idempotent, then $A = A \star A'$ since $A = A \star A = A \star A \star A'
  = A \star A'$, and so $A$ can be written $g \star g'$ with $g = A$. Now if
  $A = g \star g'$ then $A = A'$ as a consequence of
  Lemma~\ref{lem:gpoinvlem}.
\end{proof}

\begin{lemma}
  \label{lem:lridem}If $g \in G$ then there are unique idempotents $A$ and $B$
  such that $g = g \star A$ and $g = B \star g$.
\end{lemma}

\begin{proof}
  We can take $A = g' \star g$, and this is an idempotent such that $g \star A
  = g$. Conversely if $A'$ is any other element such that $g \star A' = g$,
  then $g^{- 1} \star g = g^{- 1} \star g \star A' = A'$, so $A' = A$. The
  statements about $B$ are proved similarly.
\end{proof}

\begin{proposition}
	\label{prop:gpoidcat}
  Let $G$ be a groupoid. Then there exists a groupoid category whose groupoid
  (as in Example~\ref{example:categorygroupoid}) is isomorphic to~$G$.
\end{proposition}

\begin{proof}
  Let us define a category $\mathcal{C}$ whose objects and morphisms are all
  elements of $G$. The objects are the idempotent elements of $G$. If $A, B$
  are objects, we define $\mathrm{Hom} (A, B)$ to be the set of $g \in G$ such
  that $A \star g = g$ and $g \star B = B$. By Lemma~\ref{lem:lridem},
  \[ G = \bigsqcup_{A, B \in \mathcal{C}} \mathrm{Hom} (A, B) . \]
  We must show that if $f \in \mathrm{Hom} (A, B)$ and $g \in \mathrm{Hom} (B,
  C)$ then $g \star f$ is defined and is in $\mathrm{Hom} (A, B)$. We can
  write $f = B \star f$ and $g = g \star B$, and then $g \star f = g \star B
  \star B \star f$ is defined since $B \star B$ is defined. It is clear that
  $C \star g \star f = g \star f$ and $g \star f \star A = g \star f$, so $g
  \star f \in \mathrm{Hom} (A, C)$.
  
  Now an idempotent $A$ is itself both an object of the category and a
  morphism in $\mathrm{Hom} (A, A)$; to distinguish this double role we denote
  it as $1_A$ when regarding it as a morphism. The category axioms are easily
  checked.
\end{proof}

If $G_i$ ($i \in I$) is a parametrized family of groupoids, then the disjoint
union $G = \sqcup G_i$ is naturally a groupoid. For example a disjoint union
of groups is a groupoid.

\begin{lemma}
  If $G$ is a groupoid, and if $A$ is an idempotent, then
  \[ \mathrm{Aut} (A) := \{g \in G|g \star A = A \star g = g\} \]
  is a group.
\end{lemma}

\begin{proof}
  Let us check that $\mathrm{Aut} (A)$ is closed under $\star$. If $g, h \in
  \mathrm{Aut} (A)$, then $g \star A$ and $A \star h$ are defined so $g \star
  A \star h$ is defined. This equals $g \star A \star A \star h = g \star h$.
  We leave the remaining details to the reader.
\end{proof}

\section{Yang-Baxter equation for the six-vertex model}

The six-vertex model in statistical mechanics can be described algebraically
in terms of the matrices of weights for each vertex. See {\cite{BBF11}},
Section 1 for details. We study the Yang-Baxter equation for the matrices
which arise from the six-vertex model.

\begin{definition}
	A \textit{six-vertex matrix} $u \in \operatorname{GL}_4 (\mathbb{C})$ is a $4 \times 4$
	matrix of the form (\ref{eq:sixvertex})
  where $c_1 (u)$ and $c_2 (u)$ are nonzero. Let $S$ be the set of the
  six-vertex matrices. Let $S^{\times}$ be the set of invertible elements of
  $S$. Also let $S^{\bullet}$ be the subset of $S$ in which all six
  coefficients $a_i (u)$, $b_i (u)$, $c_i (u)$ are nonzero. Let $S^{\circ} =
  S^{\times} \cap S^{\bullet}$ be the subset of elements of $S^{\bullet}$ that
  are invertible, so that furthermore $c_1 (u) c_2 (u) - b_1 (u) b_2 (u) \neq
  0$.
\end{definition}

We will require a number of different subsets of $S$. These are all locally
closed in either the Zariski or complex topologies. These are summarized in
Table~\ref{tab:menagerie}.

\

\begin{table}[h]
	\[
  \begin{tabular}{|c|l|}
    \hline
    $S$ & $c_1, c_2 \neq 0$\\
    \hline
    $S^{\times}$ & $a_1, a_2, c_1, c_2, c_1 c_2 - b_1 b_2 \neq 0$\\
    \hline
    $S^{\bullet}$ & $a_1, a_2, b_1, b_2, c_1, c_2 \neq 0$\\
    \hline
    $S^{\circ}$ & $a_1, a_2, b_1, b_2, c_1, c_2, c_1 c_2 - b_1 b_2 \neq 0$\\
    \hline
    $S_{\operatorname{ff}}$ & $c_1, c_2 \neq 0$, $N (u) = 0$\\
    \hline
    $\Omega$ & $a_1, a_2, b_1, b_2, c_1, c_2, N (u) \neq 0$; \quad$(\Omega =
    \Omega^{\circ} \cup \Omega_B)$\\
    \hline
    $\Omega^{\circ}$ & $a_1, a_2, b_1, b_2, c_1, c_2, c_1 c_2 - b_1 b_2, N (u)
    \neq 0$\\
    \hline
    $\Omega_b$ & $a_1, a_2, c_1, c_2 \neq 0, b_1 = b_2 = 0, a_1 a_2 - c_1 c_2
    = 0$\\
    \hline
    $\Omega_a$ & $b_1, b_2, c_1, c_2 \neq 0, a_1 = a_2 = 0, b_1 b_2 - c_1 c_2
    = 0$\\
    \hline
    $\Omega_B$ & $a_1, a_2, b_1, b_2, c_1, c_2, N (u) \neq 0, c_1 c_2 - b_1
    b_2 = 0$\\
		\hline
    $\overline{\Omega}$ \rule{0pt}{12pt}
		& $\Omega^{\circ} \cup \Omega_b \cup \Omega_a \cup
    \Omega_B$\\
    \hline
	\end{tabular}\]
  \caption{\label{tab:menagerie}Various subsets of the six-vertices $u \in S$,
  including the constraints that define them. All sets require $c_1 = c_1 (u)$
  and $c_2 = c_2 (u)$ to be nonzero. The sets labeled $\Omega$ (with various
  decorations) are needed for the analysis of the non-free-fermionic part of
  the Yang-Baxter groupoid, although $\Omega_a$ and $\Omega_b$ are
  free-fermionic.}
\end{table}

We will denote by $B (u) = \left( \begin{smallmatrix}
  c_1 (u) & b_1 (u)\\
  b_2 (u) & c_2 (u)
\end{smallmatrix} \right)$ the middle $2 \times 2$ block. We will also denote
\[ N (u) = a_1 (u) a_2 (u) + b_1 (u) b_2 (u) - c_1 (u) c_2 (u) = a_1 (u) a_2
   (u) - \det (B (u)) . \]
Let $V \cong \mathbb{C}^2$ with the standard basis $e_1, e_2$. We view a
six-vertex matrix as a matrix of an operator $u \in \operatorname{End} (V \otimes V)$
in basis $e_1 \otimes e_1, e_1 \otimes e_2, e_2 \otimes e_1, e_2 \otimes e_2$.
By abuse of notation, we denote by $u$ both the six-vertex matrix and the
corresponding operator.

If $a_1 (u)$ and $a_2 (u)$ are nonzero we define
\[ u^{\ast} = \left( \begin{array}{cccc}
     a_1 (u^{\ast}) &  &  & \\
     & c_1^{\ast} (u) & b_1^{\ast} (u) & \\
     & b_2^{\ast} (u) & c_2^{\ast} (u) & \\
     &  &  & a_2^{\ast} (u)
   \end{array} \right) = \left( \begin{array}{cccc}
     a_1 (u^{\ast}) &  &  & \\
     & c_2 (u) & - b_1 (u) & \\
     & - b_2 (u) & c_1 (u) & \\
     &  &  & a_2(u^{\ast})
   \end{array} \right), \]
where $a_1 (u^{\ast})$ and $a_2 (u^{\ast})$ are defined by
\[ \label{eq:austardef} a_1 (u^{\ast}) = \frac{c_1 (u) c_2 (u) - b_1 (u) b_2
   (u)}{a_1 (u)}, \qquad a_2 (u^{\ast}) = \frac{c_1 (u) c_2 (u) - b_1 (u) b_2
   (u)}{a_2 (u)} . \]
If furthermore $u$ is invertible (or equivalently if $\det (B (u)) \neq 0$)
then $u^{\ast} = \det (B (u)) u^{- 1}$. A computation shows that
\begin{equation}
  \label{eq:nustar} N (u^{\ast}) = - \frac{\det (B (u))}{a_1 (u) a_2 (u)} N
  (u) .
\end{equation}
If $u \in S^{\bullet}$ define
\begin{equation}
  \label{eq:defdel1} \Delta_1 (u) = \frac{N (u)}{a_1 (u) b_1 (u)} = \frac{a_2
  (u) - a_1^{\ast} (u)}{b_1 (u)},
\end{equation}
\begin{equation}
  \label{eq:defdel2} \Delta_2 (u) = \frac{N (u)}{a_2 (u) b_2 (u)} = \frac{a_1
  (u) - a_2^{\ast} (u)}{b_2 (u)} .
\end{equation}
Also define $\Delta : S^{\bullet} \longrightarrow \mathbb{C}^2$ by $\Delta (u)
= (\Delta_1 (u), \Delta_2 (u))$. If $u \in S^{\circ}$, then $u^{\ast} \in
S^{\bullet}$, so $\Delta (u^{\ast})$ is defined, and it follows from
(\ref{eq:nustar}) that
\begin{equation}
  \label{eq:delustar} \Delta_1 (u^{\ast}) = \frac{N (u)}{a_2 (u) b_1 (u)},
  \qquad \Delta_2 (u^{\ast}) = \frac{N (u)}{a_1 (u) b_2 (u)} .
\end{equation}
The functions $\Delta_1 (u), \Delta_2 (u), \Delta_1 (u^{\ast})$ and $\Delta_2
(u^{\ast})$ are thus regular on $S^{\bullet}$.

Let $V \cong \mathbb{C}^2$. For $u, w, v \in \operatorname{End} (V \otimes V)$, define
the \textit{Yang-Baxter commutator} on $V \otimes V \otimes V$:
\begin{equation}
  \label{eq:ybcommutator} \llbracket u, w, v \rrbracket = (u \otimes 1)  (1
  \otimes w)  (v \otimes 1) - (1 \otimes v)  (w \otimes 1)  (1 \otimes u) .
\end{equation}
Then the \textit{Yang-Baxter equation} is the identity
\begin{equation}
  \label{eq:ybe} \llbracket u, w, v \rrbracket = 0, \quad u, w, v \in
  \operatorname{End} (V \otimes V) .
\end{equation}
We call a solution $(u, w, v)$ to (\ref{eq:ybe}) \textit{normalized} if
\begin{equation}
  \label{eq:normalcond} c_1 (w) = c_1 (u)\, c_1 (v), \qquad c_2 (w) = c_2 (u)\, c_2 (v) .
\end{equation}

\begin{theorem}
  \label{thm:wybe}Assume that $u, v$ are six-vertex matrices with $a_1 (u),
  a_2 (u), a_1 (v), a_2 (v)$ all nonzero so that $u^{\ast}$ and $v^{\ast}$ are
  defined. A necessary and sufficient condition that there exists a six-vertex
  matrix $w$ with $\llbracket u, w, v \rrbracket = 0$ is that:
  \begin{equation}
    \label{eq:wcond} \frac{N (u) b_1 (v)}{a_1 (u)} = \frac{N (v) b_1 (u)}{a_2
    (v)}, \qquad \qquad \frac{N (u) b_2 (v)}{a_2 (u)} = \frac{N (v) b_2
    (u)}{a_1 (v)} .
  \end{equation}
  The solution $w$ is unique up to scalar multiple, and there is a unique
  normalized solution determined by (\ref{eq:normalcond}) and
  \begin{equation}
    \label{eq:wbcdef} \begin{array}{l}
      a_1 (w) = a_1 (u) a_1 (v) - b_2 (u) b_1 (v), \qquad a_2 (w) = a_2 (u)
      a_2 (v) - b_1 (u) b_2 (v),\\
      b_1 (w) = a_1 (u^{\ast}) b_1 (v) + b_1 (u) a_1 (v) = b_1 (u) a_2
      (v^{\ast}) + a_2 (u) b_1 (v),\\
      b_2 (w) = a_2 (u^{\ast}) b_2 (v) + b_2 (u) a_2 (v) = b_2 (u) a_1
      (v^{\ast}) + a_1 (u) b_2 (v) .
    \end{array}
  \end{equation}
  The equivalence of the alternative expressions for $b_1 (w)$ and $b_2 (w)$
  follows from~(\ref{eq:wcond}).
\end{theorem}

\begin{proof}
  There are 14 equations that must be satisfied for the matrix $\llbracket u,
  w, v \rrbracket$ to vanish. However one equation is duplicated. The
  duplicated equation is
  \begin{equation}
    \label{eq:duplicat} c_1 (u) c_2 (w) c_1 (v) = c_2 (u) c_1 (w) c_2 (v)
  \end{equation}
  The remaining twelve equations are:
  \begin{equation}
    \label{eq:thetwelve} \begin{array}{rcl}
      c_1 (u) a_1 (w) b_2 (v) + b_2 (u) c_1 (w) c_2 (v) & = & c_1 (u) b_2 (w)
      a_1 (v)\\
      b_2 (u) c_1 (w) b_1 (v) + c_1 (u) a_1 (w) c_1 (v) & = & a_1 (u) c_1 (w)
      a_1 (v)\\
      c_2 (u) a_1 (w) b_2 (v) + b_2 (u) c_2 (w) c_1 (v) & = & c_2 (u) b_2 (w)
      a_1 (v)\\
      c_2 (u) c_1 (w) b_1 (v) + b_1 (u) a_1 (w) c_1 (v) & = & a_1 (u) b_1 (w)
      c_1 (v)\\
      c_2 (u) c_1 (w) b_2 (v) + b_2 (u) a_2 (w) c_1 (v) & = & a_2 (u) b_2 (w)
      c_1 (v)\\
      b_1 (u) c_1 (w) b_2 (v) + c_1 (u) a_2 (w) c_1 (v) & = & a_2 (u) c_1 (w)
      a_2 (v)\\
      b_2 (u) c_2 (w) b_1 (v) + c_2 (u) a_1 (w) c_2 (v) & = & a_1 (u) c_2 (w)
      a_1 (v)\\
      c_1 (u) c_2 (w) b_1 (v) + b_1 (u) a_1 (w) c_2 (v) & = & a_1 (u) b_1 (w)
      c_2 (v)\\
      c_1 (u) c_2 (w) b_2 (v) + b_2 (u) a_2 (w) c_2 (v) & = & a_2 (u) b_2 (w)
      c_2 (v)\\
      c_1 (u) a_2 (w) b_1 (v) + b_1 (u) c_1 (w) c_2 (v) & = & c_1 (u) b_1 (w)
      a_2 (v)\\
      b_1 (u) c_2 (w) b_2 (v) + c_2 (u) a_2 (w) c_2 (v) & = & a_2 (u) c_2 (w)
      a_2 (v)\\
      c_2 (u) a_2 (w) b_1 (v) + b_1 (u) c_2 (w) c_1 (v) & = & c_2 (u) b_1 (w)
      a_2 (v)
    \end{array}
  \end{equation}
  Equation (\ref{eq:duplicat}) implies that if a solution $w$ exists, a unique
  constant multiple is normalized. Therefore we may impose the condition
  (\ref{eq:normalcond}). Substituting these values for $c_1 (w)$ and $c_2
  (w)$, each of the twelve equations is now divisible by one of $c_1 (u), c_2
  (u), c_1 (v)$ or $c_2 (v)$, and on dividing these away, each of the twelve
  equations occurs twice and there are only six unique equations to be
  satisfied. These are

  \[ \begin{array}{rcc}
       b_2 (u) c_1 (v) c_2 (v) - b_2 (w) a_1 (v) + a_1 (w) b_2 (v) & = & 0\\
       - a_1 (u) a_1 (v) + b_2 (u) b_1 (v) + a_1 (w) & = & 0\\
       c_1 (u) c_2 (u) b_1 (v) + b_1 (u) a_1 (w) - a_1 (u) b_1 (w) & = & 0\\
       c_1 (u) c_2 (u) b_2 (v) + b_2 (u) a_2 (w) - a_2 (u) b_2 (w) & = & 0\\
       - a_2 (u) a_2 (v) + b_1 (u) b_2 (v) + a_2 (w) & = & 0\\
       b_1 (u) c_1 (v) c_2 (v) - b_1 (w) a_2 (v) + a_2 (w) b_1 (v) & = & 0
     \end{array} \]
  From the second and fifth equation we see that $a_1 (w)$ and $a_2 (w)$ must
  have the values in (\ref{eq:wbcdef}). Substituting these values in the
  remaining four equations we obtain
  \[ \begin{array}{rcc}
       a_1 (u) a_1 (v) b_2 (v) - b_2 (u) b_1 (v) b_2 (v) + b_2 (u) c_1 (v) c_2
       (v) - b_2 (w) a_1 (v) & = & 0\\
       a_1 (u) b_1 (u) a_1 (v) - b_1 (u) b_2 (u) b_1 (v) + c_1 (u) c_2 (u) b_1
       (v) - a_1 (u) b_1 (w) & = & 0\\
       a_2 (u) b_2 (u) a_2 (v) - b_1 (u) b_2 (u) b_2 (v) + c_1 (u) c_2 (u) b_2
       (v) - a_2 (u) b_2 (w) & = & 0\\
       a_2 (u) a_2 (v) b_1 (v) - b_1 (u) b_1 (v) b_2 (v) + b_1 (u) c_1 (v) c_2
       (v) - b_1 (w) a_2 (v) & = & 0
     \end{array} \]
  The second and fourth are each equivalent one of the two expressions for
  $b_1 (w)$ in (\ref{eq:wbcdef}), and similarly the first and third equations
  are each equivalent to one of the two expressions for $b_2 (w)$. The two
  expressions for $b_1 (w)$ are equivalent if and only if (\ref{eq:wcond}) is
  satisfied, and similarly for $b_2 (w)$. We have proved that a solution
  exists if and only if (\ref{eq:wcond}) is satisfied, and if so, there is a
  unique normalized solution.
\end{proof}

\begin{remark}
  Conditions (\ref{eq:wcond}) can also be written
  \begin{equation}
    \begin{array}{c}
      (a_1 (u^{\ast}) - a_2 (u)) b_1 (v) = (a_2 (v^{\ast}) - a_1 (v)) b_1 (u)
      = 0,\\
      (a_2 (u^{\ast}) - a_1 (u)) b_2 (v) = (a_1 (v^{\ast}) - a_2 (v)) b_2 (u)
      = 0.
    \end{array}
  \end{equation}
\end{remark}

\begin{corollary}
  \label{cor:conddel}If $u, v \in S^{\bullet}$ a necessary and sufficient
  condition that there exists $w \in S$ such that $\llbracket u, w, v
  \rrbracket = 0$ is that $\Delta (u) = \Delta (v^{\ast})$. If this is
  satisfied, then $w$ is determined up to constant multiple, and may be
  normalized as in (\ref{eq:wbcdef}).
\end{corollary}

\begin{proof}
  In Theorem~\ref{thm:wybe}, we only assumed that $a_1 (u)$, $a_2 (u)$, $a_1
  (v)$ and $a_2 (v)$ are nonzero. If $u, v \in S^{\bullet}$, that is, if $b_1
  (u), b_2 (u), b_1 (v)$ and $b_2 (v)$ are nonzero, then the two equations in
  (\ref{eq:wcond}) are equivalent to $\Delta_1 (u) = \Delta_1 (v^{\ast})$ and
  $\Delta_2 (u) = \Delta_2 (u^{\ast})$.
\end{proof}

Observe that if $u, v \in S^{\bullet}$ satisfy $\Delta (u) = \Delta
(v^{\ast})$, then Theorem~\ref{thm:wybe} guarantees that there is $w \in S$
such that $\llbracket u, w, v \rrbracket = 0$, but it does not guarantee that
$w \in S^{\bullet}$. Also $S^{\bullet}$ is not closed under the map $u \mapsto
u^{\ast}$. However the set $S^{\circ}$ of invertible elements in
$S^{\bullet}$ is closed under $u \mapsto u^{\ast}$. The set $S^{\circ}$
is open in $S$, so conclusions we draw in this case hold ``generically.''

\begin{lemma}
  \label{lem:sixcases}Suppose that $u, v, w \in S^{\times}$ satisfy
  $\llbracket u, w, v \rrbracket = 0$. Then
  \[ \llbracket u, w, v \rrbracket = \llbracket u^{\ast}, v, w \rrbracket =
     \llbracket w, u, v^{\ast} \rrbracket = \llbracket v, u^{\ast}, w^{\ast}
     \rrbracket = \llbracket w^{\ast}, v^{\ast}, u \rrbracket = \llbracket
     v^{\ast}, w^{\ast}, u^{\ast} \rrbracket = 0. \]
\end{lemma}

\begin{proof}
  Note that $u, v$ and $w$ are invertible $S^{\times}$. In the identity
  \[ (u \otimes I)  (I \otimes w)  (v \otimes I) - (I \otimes v)  (w \otimes
     I)  (I \otimes u) = \llbracket u, w, v \rrbracket = 0, \]
  multiplying on the left by $u^{- 1} \otimes I$ and on the right by $I
  \otimes u^{- 1}$ gives $\llbracket u^{- 1}, w, v \rrbracket = 0$. Then
  multiplying by $\det (B (u))$ gives $\llbracket u^{\ast}, w, v \rrbracket =
  0$. The identity $\llbracket w, u, v^{\ast} \rrbracket = 0$ is proved
  similarly. Applying the operations $(u, w, v) \mapsto (u^{\ast}, v, w)$ and
  $(u, w, v) \mapsto (w, u, v^{\ast})$ gives the six identities.
\end{proof}

In the following Proposition, $\Delta_0(u):=\Delta_1(u)\Delta_2(u)$.
It follows by comparing (\ref{eq:defdel1}) and (\ref{eq:defdel2}) with
(\ref{eq:delustar}) that
\begin{equation}\label{eq:deloost}\Delta_0(u)=\Delta_0(u^\ast).
\end{equation}

\begin{proposition}
  \label{prop:threedel}If $u, w, v \in S^{\circ}$ satisfy $\llbracket u, w, v
  \rrbracket = 0$, then
  \[ \Delta (u) = \Delta (w), \qquad \Delta (u) = \Delta (v^{\ast}), \qquad
     \Delta (w^{\ast}) = \Delta (u^{\ast}) . \]
  Moreover
  \[\Delta_0(u)=\Delta_0(v)=\Delta_0(w).\]
\end{proposition}

\begin{proof}
  The first assertion follows from Theorem~\ref{thm:wybe} and Lemma~\ref{lem:sixcases}.
  The second fact follow from (\ref{eq:deloost}).
\end{proof}

If $u \in S$ and $N (u) = 0$, the vertex $u$ is called
{\textit{free-fermionic}}. Let $S_{\operatorname{ff}}$ be the set of all free-fermionic
matrices. Observe that if $u \in S_{\operatorname{ff}}$ then $a_1 (u) a_2 (u) + b_1
(u) b_2 (u)$ is automatically nonzero, since it equals $c_1 (u) c_2 (u) \neq
0$.

\begin{lemma}
  If $u$ is free-fermionic and $a_1 (u), a_2 (u)$ are nonzero, so that
  $u^{\ast}$ is defined, then
  \begin{equation}
    \label{eq:ffaster} a_1 (u^{\ast}) = a_2 (u), \qquad a_2 (u^{\ast}) = a_1
    (u) .
  \end{equation}
  The map $u \mapsto u^{\ast}$ extends to a continuous map on all
  free-fermionic elements of~$S$.
\end{lemma}

\begin{proof}
  If $a_1 (u)$ and $a_2 (u)$ are nonzero, then (\ref{eq:ffaster}) follows from
  $c_1 (u) c_2 (u) - b_1 (u) b_2 (u) = a_1 (u) a_2 (u)$. This implies
  (\ref{eq:ffaster}). Using this formula for $a_1 (u^{\ast})$ and $a_2
  (u^{\ast})$ gives the continuous extension to all free-fermionic six-vertex
  matrices.
\end{proof}

\section{The Free-Fermionic Yang-Baxter equation}\label{sec:freeferm}

If $u, v$ are free-fermionic, then Theorem~\ref{thm:wybe} simplifies as
follows. Versions of this were proved by Korepin (see {\cite{KBI93}}, page 126),
and by Brubaker, Bump and Friedberg~{\cite{BBF11}}.

\begin{theorem}[Korepin; Brubaker-Bump-Friedberg]
  Suppose $u, v \in S$ are free-fermionic. Then there exists $w \in S$ that is
  also free-fermionic, such that $\llbracket u, w, v \rrbracket = 0$. We have
  \begin{equation}
    \label{eq:ffwybe} \begin{array}{lc}
      a_1 (w) = a_1 (u) a_1 (v) - b_2 (u) b_1 (v), & a_2 (w) = - b_1 (u) b_2
      (v) + a_2 (u) a_2 (v),\\
      b_1 (w) = b_1 (u) a_1 (v) + a_2 (u) b_1 (v), & b_2 (w) = a_1 (u) b_2 (v)
      + b_2 (u) a_2 (v)
    \end{array}
  \end{equation}
\end{theorem}

\begin{proof}
  First assume that $u, v \in S^{\circ}$. When $a_1 (u), a_2 (u) \neq 0$, the
  free-fermionic condition is equivalent to $\Delta (u) = (0, 0)$. Thus by
  Corollary~\ref{cor:conddel}, condition (\ref{eq:wcond}) is satisfied and by
  Theorem~\ref{thm:wybe} there exists $w \in S$ such that $\llbracket u, w, v
  \rrbracket = 0$. Equation (\ref{eq:ffwybe}) follows from (\ref{eq:wbcdef})
  taking (\ref{eq:ffaster}) into account. The alternative expressions for $b_1
  (w)$ and $b_2 (w)$ reduce to the same formula in this case. We see that
  there exists an open subset $U$ of $S_{\operatorname{ff}} \times S_{\operatorname{ff}}$ such
  that if $(u, v) \in U$, then $w$ is in the dense open subset $S^{\circ}$ of
  $S$. In this case $\Delta (w) = \Delta (v) = (0, 0)$ by
  Propostion~\ref{prop:threedel}, so $w$ is free-fermionic. These arguments
  rely on the assumption that $(u, v)$ lies in an open subset of
  $S_{\operatorname{ff}} \times S_{\operatorname{ff}}$ but since the right-hand side of
  (\ref{eq:ffwybe}) is obviously continuous on all $S_{\operatorname{ff}} \times
  S_{\operatorname{ff}}$, the general case follows by continuity.
\end{proof}
\[ \left(\begin{array}{cc}
     a_1 (u) & - b_2 (u)\\
     b_1 (u) & a_2 (u)
   \end{array}\right) \left(\begin{array}{cc}
     a_1 (v) & - b_2 (v)\\
     b_1 (v) & a_2 (v)
   \end{array}\right) = \left(\begin{array}{cc}
     a_1 (w) & - b_2 (w)\\
     b_1 (w) & a_2 (w)
   \end{array}\right) \]
This result may be explained in terms of a parametrized Yang-Baxter equation.
Let $\mathfrak{G}_{\operatorname{ff}}$ be the group $\operatorname{GL} (2, \mathbb{C}) \times
\mathbb{C}^{\times}$ and let $R_{\operatorname{ff}} : \mathfrak{G}_{\operatorname{ff}}
\longrightarrow S_{\operatorname{ff}}$ be the bijective map
\[ R_{\operatorname{ff}} \left( \left(\begin{array}{cc}
     a_1 & - b_2\\
     b_1 & a_2
   \end{array}\right), c_1 \right) = \left(\begin{array}{cccc}
     a_1 &  &  & \\
     & c_1 & b_2 & \\
     & b_1 & c_2 & \\
     &  &  & a_2
   \end{array}\right), \qquad c_2 = \frac{a_1 a_2 + b_1 b_2}{c_1} . \]
\begin{corollary}
  The map $R_{\operatorname{ff}} : \mathfrak{G}_{\operatorname{ff}} \longrightarrow
  S_{\operatorname{ff}}$ is a parametrized Yang-Baxter equation with parameter group
  $\Gamma_{\operatorname{ff}} \cong \operatorname{GL} (2, \mathbb{C}) \times \operatorname{GL} (1,
  \mathbb{C})$. Thus if $\gamma, \delta \in \mathfrak{G}_{\operatorname{ff}}$ then
  \[ \llbracket R_{\operatorname{ff}} (\gamma), R_{\operatorname{ff}} (\gamma \delta),
     R_{\operatorname{ff}} (\delta) \rrbracket = 0. \]
\end{corollary}

\begin{proof}
  Indeed (\ref{eq:ffwybe}) implies that $w = R_{\operatorname{ff}} (\gamma \delta)$,
  where $\gamma \delta$ is just the product of $\gamma$ and $\delta$ in the
  group.
\end{proof}

\section{Preparations for the Non-Free-Fermionic Groupoid}

\

We now turn to the non-free-fermionic case. In order to obtain a groupoid
parametrized Yang-Baxter equation we will have to ``blow up'' part of the
boundary of the free-fermionic domain. We will carry out these details in the
next section, but here we make some preparations.

We have introduced several sets $\Omega^{\circ}, \Omega_b, \Omega_a$ and
$\Omega_B$ in Table~\ref{tab:menagerie}, of which $\Omega^{\circ}$ is the open
subset of $S$ defined by the nonvanishing of $a_1, a_2, b_1, b_2, c_1, c_2$,
$\det (B (u))$ and $N (u)$. We refer to the table for the definitions of the
other three sets. The set $\Omega_B$ is of codimension~1 in $S$, and the
subsets $\Omega_b$ and $\Omega_a$ are smaller, of codimension~3. The sets
$\Omega_b$ and $\Omega_a$ are free-fermionic. Elements of $\Omega^{\circ}$ and
$\Omega_B$ are not. Elements of $\Omega^{\circ}$ and $\Omega_b$ are
invertible, but elements of $\Omega_a$ and $\Omega_B$ are not. We let
$\overline{\Omega}$ be the union of the four disjoint sets $\Omega^{\circ},
\Omega_b, \Omega_a$ and~$\Omega_B$. We will also denote $\partial \Omega =
\Omega_b \cup \Omega_a \cup \Omega_B$. We will describe $\Omega^{\circ}$ as
the interior of $\overline{\Omega}$ and $\partial \Omega$ as the boundary.
Also let $\Omega = \Omega^{\circ} \cup \Omega_B$, this being the set of
matrices characterized by the inequalities $a_1, a_2, b_1, b_2, c_1, c_2, N
(u) \neq 0$.

\begin{proposition}
  \label{prop:zirr}The subset
  \[ X = \{  (u, v) \in \Omega^{\circ} \times \Omega^{\circ} | \Delta (u) =
     \Delta (v^{\ast}) \} \]
  is irreducible in the Zariski topology.
\end{proposition}

\begin{proof}
  Let $\alpha=(\alpha_1,\alpha_2) \in (\mathbb{C}^{\times})^{2}$ and let $L_{\alpha} = \{ u
  \in \Omega^{\circ} | \Delta (u) = \alpha \}$, $R_{\alpha} = \{ v \in
  \Omega^{\circ} | \Delta (v^{\ast}) = \alpha \}$. Then $X$ is the union of
  the fibers of the map $\varphi : X \longrightarrow \mathbb{C}^{\times}$
  mapping $(u, v) \in X$ to $\Delta (u) = \Delta (v^{\ast})$. Now consider
  $L_{\alpha}$. This is the locus of
  \[ \begin{array}{c}
       a_1 (u) a_2 (u) + b_1 (u) b_2 (u) - c_1 (u) c_2 (u) - \alpha_1 a_1 (u)
       b_1 (u) = 0\\
       a_1 (u) a_2 (u) + b_1 (u) b_2 (u) - c_1 (u) c_2 (u) - \alpha_2 a_2 (u)
       b_1 (u) = 0.
     \end{array} \]
  It is the intersection of two quadrics and as such it is irreducible if
  $\alpha$ is in general position. Similarly $R_{\alpha}$ is also irreducible.
  Thus the fiber $\varphi^{- 1} (\alpha)
  \cong L_{\alpha} \times R_{\alpha}$ is irreducible for $\alpha$ in general
  position. Thus $X$ is fibered over $(\mathbb{C}^{\times})^2$ with
  generically irreducible fibers, and it is therefore irreducible.
\end{proof}

\begin{proposition}
  \label{prop:odelto}Suppose $u, v$ are elements of $\Omega^{\circ}$ such that
  $\Delta (u) = \Delta (v^{\ast})$. Let $w$ be the normalized solution to the
  $\llbracket u, w, v \rrbracket = 0$ as in Theorem~\ref{thm:wybe}. Then
  \begin{equation}
    \label{eq:dvwid} a_1 (v) b_1 (v) N (w) = a_1 (w) b_1 (w) N (v), \qquad a_2
    (v) b_2 (v) N (w) = a_2 (w) b_2 (w) N (v)
  \end{equation}
  and
  \begin{equation}
    \label{eq:duwid} a_1 (w) b_2 (w) N (u) = a_1 (u) b_2 (u) N (w), \qquad a_2
    (w) b_1 (w) N (u) = a_2 (u) b_1 (u) N (w),
  \end{equation}
\end{proposition}

\begin{remark}
  Beginning with this Proposition, we will make use of continuity arguments in
  the sequel. This is for convenience and efficiency and could be replaced by
  calculations with equations. To emphasize this point, we will give two
  proofs of this result, one using a continuity argument, and one arguing
  directly from the equations.
\end{remark}

\begin{proof}[First Proof]
  First assume that $w \in \Omega^{\circ}$. Then these identities follow from
  Proposition~\ref{prop:threedel} and equations (\ref{eq:defdel1}),
  (\ref{eq:defdel2}) and (\ref{eq:delustar}). We will deduce the general case
  of (\ref{eq:dvwid}) from this special case by continuity. Let $X$ be as in
  Proposition~\ref{prop:zirr}. The subset of $X$ where $w \in \Omega^{\circ}$
  is open in $X$. But $X$ is irreducible, so every nonempty open set is dense.
  Thus (\ref{eq:dvwid}) and (\ref{eq:duwid}) are true on all of $X$ by
  continuity.
\end{proof}

\begin{proof}[Second Proof]
  Alternative to the continuity argument in the first proof, we may modify the
  reasoning in the proof of Theorem~\ref{thm:wybe} and deduce a formula for
  $u$ in terms of $w$ and $v$. It is important that the argument is valid
  without assuming that $w \in \Omega^{\circ}$. We rewrite the normalization
  condition (\ref{eq:normalcond}) as
  \[ c_1 (u) = c_1 (w) / c_1 (v), \quad c_2 (u) = c_2 (w) / c_1 (v) . \]
  Substituting this into the equations (\ref{eq:thetwelve}), clearing the
  denominators and eliminating the redundant equations, we obtain
  \[ \begin{array}{rcc}
       b_2 (u) c_1 (v) c_2 (v) - b_2 (w) a_1 (v) + a_1 (w) b_2 (v) & = & 0\\
       - a_1 (u) a_1 (v) + b_2 (u) b_1 (v) + a_1 (w) & = & 0\\
       - b_1 (u) a_1 (w) c_1 (v) c_2 (v) + a_1 (u) b_1 (w) c_1 (v) c_2 (v) -
       c_1 (w) c_2 (w) b_1 (v) & = & 0\\
       - b_2 (u) a_2 (w) c_1 (v) c_2 (v) + a_2 (u) b_2 (w) c_1 (v) c_2 (v) -
       c_1 (w) c_2 (w) b_2 (v) & = & 0\\
       - a_2 (u) a_2 (v) + b_1 (u) b_2 (v) + a_2 (w) & = & 0\\
       b_1 (u) c_1 (v) c_2 (v) - b_1 (w) a_2 (v) + a_2 (w) b_1 (v) & = & 0
     \end{array} \]
  Since $v \in \Omega^{\circ}$, $a_1 (v)$ and $a_2 (v)$ are nonzero so
  \[ a_1 (u) = \frac{b_2 (u) b_1 (v) + a_1 (w)}{a_1 (v)}, \qquad a_2 (u) =
     \frac{b_1 (u) b_2 (v) + a_2 (w)}{a_2 (v)} . \]
  Substituting these there are eight equations but only four nonredundant
  ones. Among these are two that may be rearranged as:
  \[ b_1 (u) = \frac{b_1 (w) a_2 (v) - a_2 (w) b_1 (v)}{c_1 (v) c_2 (v)},
     \qquad b_2 (u) = \frac{b_2 (w) a_1 (v) - a_1 (w) b_2 (v)}{c_1 (v) c_2
     (v)} . \]
  Substituting these values, only two nonredundant equations remain, and these
  are (\ref{eq:dvwid}). Equations (\ref{eq:duwid}) may be proved the same way,
  solving for $v$ instead of~$u.$
\end{proof}

\begin{lemma}
  \label{lem:omegaab}Suppose that $u, v$ are elements of $\Omega^{\circ}$ such
  that $\Delta (u) = \Delta (v^{\ast})$. Let $w$ be the normalized solution to
  the $\llbracket u, w, v \rrbracket = 0$ as in Theorem~\ref{thm:wybe}. Assume
  that $N (w) = 0$. Then either $w \in \Omega_a$ or $w \in \Omega_b$.
\end{lemma}

\begin{proof}
  Since $u, v \in \Omega^{\circ}$, we have $N (u) .N (v) \neq 0$. Since $N (w)
  = 0$, equations (\ref{eq:dvwid}) and (\ref{eq:duwid}) imply that $a_1 (w)
  b_1 (w)$, $a_1 (w) b_2 (w)$, $a_2 (w) b_1 (w)$ and $a_2 (w) b_2 (w) = 0$.
  Therefore either $a_1 (w) = a_2 (w) = 0$ or $b_1 (w) = b_2 (w) = 0$. Suppose
  that $b_1 (w) = b_2 (w) = 0$. Then $N (w) = a_1 (w) a_2 (w) - c_1 (w) c_2
  (w)$ and since $N (w)$ and $c_1 (w) c_2 (w)$ both vanish, we must have $a_1
  (w)$ and $a_2 (w)$ both nonzero and $w \in \Omega_b$. In the other case
  where $a_1 (w) = a_2 (w) = 0$ similar reasoning shows that $w \in \Omega_b$.
\end{proof}

\begin{proposition}
  If $u, v \in \Omega^{\circ}$ such that $\Delta (u) = \Delta (v^{\ast})$, and
  if $w$ is the normalized solution to $\llbracket u, w, v \rrbracket = 0$
  guaranteed by Theorem~\ref{thm:wybe}. Then $w \in \overline{\Omega}$.
\end{proposition}

\begin{proof}
  We must prove that $w$ is in one of the sets $\Omega^{\circ}$, $\Omega_b$,
  $\Omega_a$ or $\Omega_B$. We may assume that $w \notin \Omega^{\circ}$ so at
  least one of $a_1 (w)$, $a_2 (w)$, $b_1 (w)$, $b_2 (w)$, $c_1 (w) c_2 (w) -
  b_1 (w) b_2 (w)$ or $N (w)$ is zero. The cases where $N (w) = 0$ are handled
  by Lemma~\ref{lem:omegaab}. Therefore we assume that $N (w) \neq 0$. The
  left side of (\ref{eq:dvwid}) and (\ref{eq:duwid}) is then nonzero and so
  these equations imply that $a_1 (w), a_2 (w), b_1 (w)$ and $b_2 (w)$ are all
  nonzero. Therefore $c_1 (w) c_2 (w) - b_1 (w) b_2 (w) = 0$ and $w \in
  \Omega_B$.
\end{proof}

\begin{proposition}
  \label{prop:bdyuv}Suppose that $u, v \in \overline{\Omega}$ and at least one
  of $u, v \in \Omega_b \cup \Omega_a$. Then there exists $w \in \Omega$ such
  that $\llbracket u, w, v \rrbracket = 0$. The solution is unique unless both
  $u, v \in \Omega_a$.
\end{proposition}

\begin{proof}
  We assume that $u \in \Omega_b \cup \Omega_a$, leaving the other case 
  (where $v\in\Omega_b\cup\Omega_a$) to the reader. 

  If $u \in \Omega_b$, then solving the Yang-Baxter equation as in the proof
  of Theorem~\ref{thm:wybe} leads to the unique normalized solution
  \[ \begin{array}{lll}
       a_1 (w) = a_1 (u) a_1 (v), & \hspace{1.2em} & a_2 (w) = a_2 (u) a_2
       (v),\\
       b_1 (w) = a_2 (u) b_1 (v), &  & b_2 (w) = a_1 (u) b_2 (v) .\\
       c_1 (w) = c_1 (u) c_1 (v) . &  & c_2 (w) = c_2 (u) c_2 (v) .
     \end{array} \]
  With these values, there is only one remaining equation $a_1 (u) a_2 (u) =
  c_1 (u) c_2 (u)$, which is automatic since $u \in \Omega_b$.
  
  If $u \in \Omega_a$, then we obtain, we obtain
  \[ \begin{array}{lll}
       a_1 (w) = - b_2 (u) b_1 (v), & \hspace{1.2em} & a_2 (w) = - b_1 (u) a_2
       (v),\\
       c_1 (w) = c_1 (u) c_1 (v) . &  & c_2 (w) = c_2 (u) c_2 (v) .
     \end{array} \]
  Substituting these values, three equations remain. One is the condition $b_1
  (u) b_2 (u) - c_1 (u) c_2 (u) = 0$, which is automatically satisfied since
  $u \in \Omega_a$. The other two equations are
  \[ \begin{array}{l}
       b_2 (u) (c_1 (v) c_2 (v) - b_1 (v) b_2 (v)) = b_2 (w) a_1 (v),\\
       b_1 (u) (c_1 (v) c_2 (v) - b_1 (v) b_2 (v)) = b_1 (w) a_2 (v) .
     \end{array} \]
  If $v \notin \Omega_a$, these conditions determine $b_1 (w)$ and $b_2 (w)$
  uniquely. On the other hand if $v \in \Omega_a$ then both sides vanish, so
  $b_1 (w)$ and $b_2 (w)$ are unconstrained.
\end{proof}

\section{The Six-Vertex Yang-Baxter Groupoid}

Roughly, the idea is to define a composition law on six-vertex matrices by
writing $w = u \star v$ if $\llbracket u, w, v \rrbracket = 0$. Usually this
determines $w$ up to a constant, which we can fix by requiring the solution to
be normalized. For the groupoid inverse, which we will denote by $u'$, we can
take an appropriate multiple of the matrix inverse $u^{- 1}$. More precisely,
we find $u^{\ast}$ to be more convenient to work with, and we define $u' =
\frac{1}{c_1 (u) c_2 (u)} u^{\ast}$. These definitions must be modified in
special cases: for example $u$ may not be invertible but we still want $u'$ to
be defined.

The above description is only approximately correct. We can begin by dividing
the groupoid into the disjoint union of two parts, the free-fermionic
groupoid, and the non-free-fermionic groupoid. These may be handled
separately. Actually the free-fermionic part is a group, isomorphic to
$\operatorname{GL} (2, \mathbb{C}) \times \operatorname{GL} (1, \mathbb{C})$, and we have
already treated it in Section~\ref{sec:freeferm}. Thus it remains to construct
the non-free-fermionic groupoid $\mathfrak{G}_{\operatorname{nf}}$. As a set, we may
start with $\overline{\Omega} = \Omega^{\circ} \cup \Omega_b \cup \Omega_a
\cup \Omega_B$. Unfortunately, we need $\Delta_1, \Delta_2$ and $u^{\ast}$ to
be defined on the whole groupoid, but $\Delta$ is undefined on $\Omega_b$ and
$\Omega_a$, because the numerator and denominator in (\ref{eq:defdel1}) and
(\ref{eq:defdel2}) both vanish. Moreover $u^{\ast}$ is undefined on $\Omega_a$
because the numerator and denominator in (\ref{eq:austardef}) both vanish.

To fix these problems, we ``blow up'' $\Omega_b$ and $\Omega_a$. We define
$\mathfrak{G}_{\operatorname{nf}}$ as a set to be
\[ \left\{  (u, d_1, d_2) \in \overline{\Omega} \times \mathbb{C}^{\times}
   \times \mathbb{C}^{\times} |a_1 (u) b_1 (u) d_1 = a_2 (u) b_2 (u) d_2 = N
   (u) \right\} . \]
The map $\pi : \mathfrak{G}_{\operatorname{nf}} \longrightarrow \overline{\Omega}$ is
the projection on the first component. Note that over $\Omega = \Omega^{\circ}
\cup \Omega_B$, the fibers of $\pi$ have cardinality~1. That is, there is a
unique section $\mathbf{s}: \Omega \longrightarrow \mathfrak{G}_{\operatorname{nf}}$
that maps $u \in \Omega$ to $\mathbf{s} (u) = (u, \Delta_1 (u), \Delta_2
(u))$ and $\pi^{- 1} (u) = \{ \mathbf{s} (u) \}$.

On the other hand, suppose that $u$ is in either $\Omega_a$ or $\Omega_b$.
Then $\pi^{- 1} (u)$ consists of a 2-dimensional torus $\{ (u, d_1, d_2) |d_i
\in \mathbb{C}^{\times} \}$. This is because $a_1 (u) b_1 (u)$, $a_2 (u) b_2
(u)$ and $N (u)$ are all zero, so there is no constraint on $d_1$ and $d_2$.

We will adopt the following convention. If $\mathbf{u} \in
\mathfrak{G}_{\operatorname{nf}}$ we will use the same letter $u$ to denote the
element $\pi (\mathbf{u}) \in \overline{\Omega}$. If $u \in \Omega =
\Omega^{\circ} \cup \Omega_B$, we may even conflate $u \in \Omega$ with
$\mathbf{s} (u)$. Thus we are identifying $\Omega$ with its corresponding
subset of $\mathfrak{G}_{\operatorname{nf}}$. Therefore we consider $\Omega^{\circ}$
and $\Omega_B$ to be subsets of $\mathfrak{G}_{\operatorname{nf}}$.

The space $\Omega$ is 6-dimensional, and $\Omega_B$ is 5-dimensional. On the
other hand, $\Omega_b$ and $\Omega_a$ are both 3-dimensional. But we define
$\Gamma_b = \pi^{- 1} (\Omega_b)$ and $\Gamma_a = \pi^{- 1} (\Omega_a)$ and
these are 5-dimensional. So all three ``boundary components'' $\Gamma_b$,
$\Gamma_a$ and $\Omega_B$ are of codimension~1. We will see in
Proposition~\ref{prop:gpostarcontin} that the involution interchanges
$\Gamma_a$ and $\Omega_B$, so even though their constructions are quite
different, they are in that sense equivalent.

Our general strategy is to prove things generically for elements of the dense
open set $\Omega^{\circ}$, then deduce them by continuity in general. We have
already seen an example of this in the first proof of Proposition
\ref{prop:odelto}. Now $\mathfrak{G}_{\operatorname{nf}}$ is naturally a quasi-affine
algebraic variety, irreducible by Proposition~\ref{prop:zirr}, with
$\Omega^{\circ}$ a dense open set. But in this section we will switch to the
complex topology for our continuity arguments. So let us say what it means for
a sequence $\{ r_n \} \subset \Omega^{\circ}$ to converge to an element
$\mathbf{r}= (r, d_1, d_2)$, possibly of one of the boundary components
$\Omega_b$ or $\Omega_a$. It means that $r_n \rightarrow r$ in the complex
topology, and moreover $\Delta_1 (r_n) \rightarrow d_1$ and $\Delta_2 (r_n)
\rightarrow d_2$. It is not hard to see that $\Omega^{\circ}$ is dense in
$\mathfrak{G}_{\operatorname{nf}}$.

\begin{proposition}
  \label{prop:gpostarcontin}The map $r \mapsto r^{\ast}$ can be extended
  uniquely to a continuous map $\mathfrak{G}_{\operatorname{nf}} \longrightarrow
  \mathfrak{G}_{\operatorname{nf}}$. The map preserves $\Omega^{\circ}$ and $\Gamma_b$
  but interchanges $\Gamma_a$ and $\Omega_B$. If $\mathbf{r} = (r, d_1, d_2)$
  then $\mathbf{r}^{\ast} = (r^{\ast}, d_1^{\ast}, d_2^{\ast})$ where
  \begin{equation}
    \label{eq:gpobdyrstar} r^{\ast} = \left( \begin{array}{cccc}
      a_1 (\mathbf{r}^{\ast}) &  &  & \\
      & c_2 (r) & - b_1 (r) & \\
      & - b_2 (r) & c_1 (r) & \\
      &  &  & a_2 (\mathbf{r}^{\ast})
    \end{array} \right)
  \end{equation}
  with
  \[ (a_1 (\mathbf{r}^{\ast}), a_2 (\mathbf{r}^{\ast})) = \left\{
     \begin{array}{ll}
       (a_2 (r), a_1 (r)) & \text{if } \mathbf{r} \in \Gamma_b,\\
       (- d_1 b_1 (r), - d_2 b_2 (r)) & \text{if $\mathbf{r} \in \Gamma_a$,}\\
       (0, 0) & \text{if } \mathbf{r} \in \Omega_B
     \end{array} \right. \]
  and
  \[ (d_1^{\ast}, d_2^{\ast}) = \left\{ \begin{array}{ll}
       \left( \frac{b_2 (r)}{b_1 (r)} d_2, \frac{b_1 (r)}{b_2 (r)} d_1 \right)
       & \text{if $\mathbf{r} \in \Gamma_a$,}\\
       \left( \frac{a_1 (r)}{a_2 (r)} d_1, \frac{a_2 (r)}{a_1 (r)} d_2 \right)
       & \text{otherwise.}
     \end{array} \right. \]
\end{proposition}

\begin{proof}
  We must consider the extension to $\Gamma_b$, $\Gamma_a$ and $\Omega_B$. We
  consider $\mathbf{r} = (r, d_1, d_2)$ in one of these boundary components,
  and a sequence $\{r_n \} \subset \Omega^{\circ}$ that converges to
  $\mathbf{r}$.
  
  If $r \in \Omega_B$ then $a_1^{\ast}$ and $a_2^{\ast}$ are continuous at $r$
  and converge to zero. For the other two cases, note that
  \[ a_1^{\ast} (r_n) = \frac{a_1 (r_n) a_2 (r_n) - N (r_n)}{a_1 (r_n)} = a_2
     (r_n) - b_1 (r_n) \Delta_1 (r_n) . \]
  Thus
  \[ a_1^{\ast} (r_n) \longrightarrow \left\{ \begin{array}{ll}
       a_2 (r) & \text{if $r \in \Omega_b$,}\\
       - D_1 b_1 (r) & \text{if } r \in \Omega_a .\\
       0 & \text{if } r \in \Omega_B,
     \end{array} \right.  \qquad a_2^{\ast} (r_n) \longrightarrow \left\{
     \begin{array}{ll}
       a_1 (r) & \text{if $r \in \Omega_b$,}\\
       - D_2 b_2 (r) & \text{if } r \in \Omega_a .\\
       0 & \text{if } r \in \Omega_B,
     \end{array} \right. \]
  We must also consider the limits of the $\Delta_i (r_n^{\ast})$. If $r \in
  \Omega_b$ or $\Omega_B$ then
  \[ \Delta_1 (r_n^{\ast}) = \frac{a_1 (r_n)}{a_2 (r_n)} \Delta_1 (r_n)
     \longrightarrow \frac{a_1 (r)}{a_2 (r)} D_1 . \]
  Let us assume that $r \in \Omega_a$. Then $r_n^{\ast} \rightarrow r^{\ast}$
  where $r^{\ast}$ is already computed. We have $N (r^{\ast}) = a_1 (r^{\ast})
  a_2 (r^{\ast}) - \det (B (r^{\ast})) = a_1 (r^{\ast}) a_2 (r^{\ast})$ so
  $\Delta_1 (r^{\ast}) = a_2 (r^{\ast}) / (- b_1 (r)) = D_2 b_2 (r) / b_1
  (r)$.
\end{proof}

\begin{corollary}
  We have
  \[ \Delta (\mathbf{r}^{\ast}) = \left\{\begin{array}{ll}
       \left( \frac{a_1 (r)}{a_2 (r)} \Delta_1 (\mathbf{r}), \frac{a_1
       (r)}{a_2 (r)} \Delta_2 (\mathbf{r}) \right) & \text{if $\mathbf{r}
       \notin \Gamma_a$},\\
       \left( \frac{b_2 (r)}{b_1 (r)} \Delta_2 (\mathbf{r}), \frac{b_1
       (r)}{b_2 (r)} \Delta_1 (\mathbf{r}) \right) & \operatorname{if} \mathbf{r}
       \in \Gamma_a .
     \end{array}\right. \]
\end{corollary}

\begin{corollary}
  The map $r \mapsto r'$, where $r' = \frac{1}{c_1 (r) c_2 (r)} r^{\ast}$
  extends to a continuous map $\mathbf{r} \rightarrow \mathbf{r}'$ of
  $\mathfrak{G}_{\operatorname{nf}}$.
\end{corollary}

\begin{proof}
  Since $c_1 (\mathbf{r})$ and $c_2 (\mathbf{r})$ are defined on the
  entire groupoid.
\end{proof}

\begin{lemma}
	\label{lem:cxdense}
  Let $\mathbf{r} \in \mathfrak{G}_{\operatorname{nf}}$. Then there is a sequence
  $\{ r_n \} \subset \Omega^{\circ}$ such that $r_n \rightarrow \mathbf{r}$,
  with $\Delta (r_n) = \Delta (\mathbf{r})$ and $\Delta (r_n^{\ast}) =
  \Delta (\mathbf{r})$.
\end{lemma}

\begin{proof}
  If $\mathbf{r}= (r, d_1, d_2) \in \Omega^{\circ}$ we may take $r_n = r$,
  so this case is obvious. We need to check this if $\mathbf{r}$ is in one
  of the boundary components.
  
  First suppose $\mathbf{r} \in \Gamma_b$, so $b_1 (r) = b_2 (r) = 0$ and
  $a_1 (r) a_2 (r) = c_1 (r) c_2 (r)$. We will choose the sequence $r_n$ so that
  $a_1 (r_n) = a_1 (r)$ and $a_2 (r_n) = a_2 (r)$. Let $(d_1, d_2) = \Delta
  (\mathbf{r})$. In order that $\Delta (r_n) = (d_1, d_2)$ we need:
  \begin{equation}
    \label{eq:nabrats} \frac{N (r_n)}{a_1 (r) b_1 (r_n)} = d_1, \qquad \frac{N
    (r_n)}{a_2 (r) b_2 (r_n)} = d_2 .
  \end{equation}
  Thus we want $b_1 (r_n), b_2 (r_n) \rightarrow 0$ preserving the ratio
  \[ \frac{b_1 (r_n)}{b_2 (r_n)} = \frac{d_2 a_2 (r)}{d_1 a_1 (r)} . \]
  Then if the first equation in (\ref{eq:nabrats}) is satisfied, the second is
  automatically true.
  
  If $c$ denotes the constant on the right, then we can choose a sequence $b_2
  (r_n) \rightarrow 0$ and define $b_1 (r_n) = c \cdot b_2 (r_n)$. Now we
  need:
  \[ a_1 (r) a_2 (r) + b_1 (r_n) b_2 (r_n) - c_1 (r_n) c_2 (r_n) = d_1 a_1 (r)
     b_1 (r_n) . \]
  We can fix $c_1 (r_n) = c_1 (r)$ and then solve this equation for $c_2
  (r_n)$. The right-hand side converges to zero as does the second term on the
  left, so
  \[ a_1 (r) a_2 (r) - c_1 (r) c_2 (r_n) \rightarrow 0. \]
  Since $a_1 (r) a_2 (r) = c_1 (r) c_2 (r)$ it follows that $c_2 (r_n)
  \rightarrow c_2 (r)$.
  
  We also need to know that $\Delta (r_n^{\ast}) = \Delta
  (\mathbf{r}^{\ast})$. Indeed
	\[ \Delta_1 (\mathbf{r}^{\ast}) = \frac{a_1 (r)}{a_2 (r)} \Delta_1 (\mathbf{r}) =
     \frac{a_1 (r_n)}{a_2 (r_n)} \Delta_1 (r_n) = \Delta_1 (r_n^{\ast}) \]
  and similarly for $\Delta_2$.
  
  We leave the cases where $\mathbf{r} \in \Gamma_a$ or $\Omega_B$ to the
  reader, except to note that the cases are equivalent since if $r_n
  \rightarrow \mathbf{r}$ then $r_n^{\ast} \rightarrow \mathbf{r}^{\ast}$
  with $\mathbf{r} \in \Gamma_a$ and $\mathbf{r}^{\ast} \in \Omega_B$, so
  one sequence works for both cases.
\end{proof}

\begin{proposition}
  \label{prop:defgensta}Let $\mathbf{u}, \mathbf{v} \in
  \mathfrak{G}_{\operatorname{nf}}$. Assume that $\Delta (\mathbf{u}) = \Delta
  (\mathbf{v}^{\star})$. Then there is a unique element $\mathbf{w}$ such
  that $\llbracket u, w, v \rrbracket = 0$ is a normalized solution, and such
  that $\Delta (\mathbf{w}) = \Delta (\mathbf{v})$ and $\Delta
  (\mathbf{w}^{\ast}) = \Delta (\mathbf{u}^{\ast})$. The element
  $\mathbf{w}$ depends continuously on $\mathbf{u}$ and $\mathbf{v}$.
\end{proposition}

\begin{proof}
  This follows from Theorem~\ref{thm:wybe} if $\mathbf{u}, \mathbf{v} \in
  \Omega$, but we must consider the case where one or both is in $\Omega_b$ or
  $\Omega_a$. In this case, the existence of $w$ such that $\llbracket u, w, v
  \rrbracket = 0$ is guaranteed by Proposition~\ref{prop:bdyuv}. However if
  both $\mathbf{u}, \mathbf{v} \in \Gamma_a$ that Proposition shows that
  while $a_1 (\mathbf{w}), a_2 (\mathbf{w}), c_1 (\mathbf{w}), c_2
  (\mathbf{w})$ are determined, but we must deduce the values for $b_1
  (\mathbf{w})$ and $b_2 (\mathbf{w})$ from the requirement that $\Delta
  (\mathbf{w}) = \Delta (\mathbf{v})$. We consider the identities
  \[ b_1 (r) = \frac{a_2 (r) - a_1 (r^{\ast})}{\Delta_1 (r)}, \qquad b_1 (r) =
     \frac{a_1 (r) - a_2 (r^{\ast})}{\Delta_2 (r)}, \]
  which follow from the definitions of $\Delta_i$ and $r^{\ast}$ for $r \in
  \Omega$, and by continuity to $\mathfrak{G}_{\operatorname{nr}}$. Since we require
  $\Delta (\mathbf{w}) = \Delta (\mathbf{v})$, we see that we must define
  \[ b_1 (\mathbf{w}) = \frac{a_2 (\mathbf{w}) - a_1
     (\mathbf{w}^{\ast})}{\Delta_1 (\mathbf{v})}, \qquad b_2
     (\mathbf{w}) = \frac{a_1 (\mathbf{w}) - a_2
     (\mathbf{w}^{\ast})}{\Delta_2 (\mathbf{v})} . \]
  The continuity of the $\star$ operation is clear.
\end{proof}

Now let us define the groupoid composition $\star$ on $\mathfrak{G}_{\operatorname{nf}}$.
If $\Delta (\mathbf{u}) = \Delta (\mathbf{v}^{\ast})$, let $\mathbf{u}
\star \mathbf{v}=\mathbf{w}$, where $\mathbf{w}$ is as in
Proposition~\ref{prop:defgensta}. In this case we say that $\mathbf{u} \star
\mathbf{v}$ is {\textit{defined}}.

\begin{proposition}[Associativity for general position]
  \label{prop:gpogenasso}If $r, s, t, u, v \in \Omega^{\circ}$ are such that
  $r \star t = s$ and $t \star v = u$. Then $s \star v$ and $r \star u$ are
  both defined, and they are equal.
\end{proposition}

\begin{proof}
  Note that although by assumption $r, s, t, u, v$ are all in $\Omega^{\circ}$
  the compositions $s \star v$ and $r \star u$ could be in a boundary
  component.
  
  By Proposition~\ref{prop:threedel}, $\Delta (r) = \Delta (t^{\ast})$ and
  $\Delta (s) = \Delta (t)$, and similarly $\Delta (t) = \Delta (v^{\ast})$,
  and $\Delta (u) = \Delta (v)$. Note that $\Delta (s) = \Delta (t) = \Delta
  (v^{\ast})$ implies that $s \star v$ is defined by
  Proposition~\ref{prop:defgensta}. On the other hand, we note that
  $\llbracket t, u, v \rrbracket = 0$ since $u = t \star v$, we also have
  $\llbracket v, t^{\ast}, u^{\ast} \rrbracket = 0$. Thus $\Delta (t^{\ast}) =
  \Delta (u^{\ast})$. Therefore $\Delta (r) = \Delta (u^{\ast})$ and so $r
  \star u$ is defined.
  
  It remains to be shown that $r \star u = s \star v$. We have
  \[ c_1  (r \star u) = c_1 (r) c_1 (t) \star c_1 (v) = c_1  (s \ast v), \]
  and similarly for $c_2$. Using (\ref{eq:wbcdef}) we have
  \[ \begin{array}{lll}
       a_1  (s \star v) & = & a_1 (s) a_1 (v) - b_2 (s) b_1 (v)\\
       & = & (a_1 (r) a_1 (t) - b_2 (r) b_1 (t)) a_1 (v) - (a_1 (r) b_2 (t) +
       b_2 (r) a_1^{\ast} (t)) b_1 (v)\\
       & = & a_1 (r)  (a_1 (t) a_1 (v) - b_2 (t) b_1 (v)) - b_2 (r)  (b_1 (t)
       a_1 (v) + a_1 (t^{\ast}) b_1 (v))\\
       & = & a_1 (r) a_1 (u) - b_2 (r) b_1 (u) = a_1  (r \star u)
     \end{array} \]
  and similarly $a_2  (s \star v) = a_2  (r \star u)$.
  
  We note that $s^{\ast} = t^{\ast} \star r^{\ast}$ by
  Lemma~\ref{lem:sixcases}. Therefore using (\ref{eq:wbcdef}) and the fact
  that $b_1 (r^{\ast}) = - b_1 (r)$ and $b_2 (t^{\ast}) = - b_2 (t)$ we have
  \[ a_1 (s^{\ast}) = a_1 (r^{\ast}) a_1 (t^{\ast}) - b_1 (r^{\ast}) b_2
     (t^{\ast}) = a_1 (r^{\ast}) a_1 (t^{\ast}) - b_1 (r) b_2 (t) . \]
  Using these equations
  \[ \begin{array}{lll}
       b_1  (s \star v) & = & a_1 (s^{\ast}) b_1 (v) + b_1 (s) a_1 (v)\\
       & = & (a_1 (t^{\ast}) a_1 (r^{\ast}) - b_1 (r) b_2 (t)) b_1 (v) + (b_1
       (r) a_1 (t) + b_1 (t) a_1 (r^{\ast})) a_1 (v),\\
       &  & \\
       b_1  (r \star u) & = & a_1 (r^{\ast}) b_1 (u) + b_1 (r) a_1 (u)\\
       & = & a_1 (r^{\ast})  (a_1 (t^{\ast}) b_1 (v) + b_1 (t) a_1 (v)) + b_1
       (r)  (a_1 (t) a_1 (v) - b_2 (t) b_1 (v)) .
     \end{array} \]
  These are term by term equal, and similarly $b_2  (s \star v) = b_2  (r
  \star u)$. We see that $s \star v$ and $r \star u$ have the same Boltzmann
  weights, and so they are equal.
\end{proof}

The next two results establish that the composition law that we have defined
on $\mathfrak{G}_{\operatorname{nf}}$ satisfies the groupoid axioms.

\begin{theorem}[Associativity Axiom]
	\label{thm:assoax}
  Let $\mathbf{r}, \mathbf{s}, \mathbf{t}, \mathbf{u}, \mathbf{v}
  \in \mathfrak{G}_{\operatorname{nf}}$ such that $\mathbf{r} \star
  \mathbf{t}=\mathbf{s}$ and $\mathbf{t} \star
  \mathbf{v}=\mathbf{u}$. Then \ are defined, and they are equal.
\end{theorem}

\begin{proof}
  We could argue similarly to Proposition~\ref{prop:gpogenasso}, on a
  case-by-case basis to handle the edge cases, where one or more of
  $\mathbf{r}, \mathbf{s}, \mathbf{t}, \mathbf{u}, \mathbf{v}$ is in
  one of the boundary components. Instead, we will argue by continuity. 
  By Lemma~\ref{lem:cxdense} we may
  find sequences $\{ r_n \}, \{ t_n \}, \{ v_n \} \subset \Omega^{\circ}$ such
  that $\Delta (r_n) = \Delta (\mathbf{r})$ with $r_n \rightarrow
  \mathbf{r}$, etc. Then for each $n$, $\Delta (r_n) = \Delta (t_n^{\ast})$
  so $r_n \star t_n$ is defined; call this $s_n$. With $\{ t_n \}$ fixed, we
  may perturb the sequence $\{ r_n \}$ so that $s_n \in \Omega^{\circ}$.
  Similarly, with $\{ t_n \}$ fixed, we may purturb the sequence $\{ v_n \}$
  so that $u_n \in \Omega^{\circ}$. Now by Proposition~\ref{prop:gpogenasso},
  $r_n \star u_n$ and $s_n \star v_n$ are defined and equal, and taking the
  limit gives $\mathbf{s} \star \mathbf{v}=\mathbf{r} \star
  \mathbf{u}$.
\end{proof}

If $\mathbf{u}$, we define the groupoid inverse to be $\mathbf{u}' :=
\frac{1}{c_1 (u) c_2 (u)} \mathbf{u}^{\ast}$.

\begin{proposition}[Idempotents]
  \label{prop:idempotents}Let $d_1, d_2 \in \mathbb{C}^{\times}$. Let
  $\mathbf{I}_{d_1, d_2} = (I_{V \otimes V}, d_1, d_2)$. Then
  $\mathbf{I}_{d_1, d_2}$ is an idempotent in that $\mathbf{I}_{d_1, d_2}
  \star \mathbf{I}_{d_1, d_2} =\mathbf{I}_{d_1, d_2}$. We have
  $\mathbf{I}_{d_1, d_2}^{\ast} =\mathbf{I}_{d_1, d_2}'
  =\mathbf{I}_{d_1, d_2}$. Furthermore $\Delta (\mathbf{I}_{d_1, d_2}) =
  (d_1, d_2)$. If $\mathbf{u} \star \mathbf{I}_{d_1, d_2}$ is defined
  $\mathbf{u} \star \mathbf{I}_{d_1, d_2} =\mathbf{u}$. If
  $\mathbf{I}_{d_1, d_2} \star \mathbf{v}$ is defined then
  $\mathbf{I}_{d_1, d_2} \star \mathbf{v}=\mathbf{v}$.
\end{proposition}

\begin{proof}
  Proposition~\ref{prop:gpostarcontin} shows that $\mathbf{I}_{d_1,
  d_2}^{\ast} =\mathbf{I}_{d_1, d_2}$. Therefore $\mathbf{I}_{d_1, d_2}
  \star \mathbf{I}_{d_1, d_2}$ is defined. It is of the form $\mathbf{J}=
  (J, D_1, D_2)$, where by the uniqueness assertion in
  Proposition~\ref{prop:bdyuv}, $J$ is the {\textit{unique}} element of
  $\overline{\Omega}$ such that $\llbracket I, J, I \rrbracket = 0$ is a
  normalized solution; thus $J = I$, and by Proposition~\ref{prop:defgensta}
  we have $(D_1, D_2) = \Delta (\mathbf{i}) = \Delta (I_{d_1, d_2}) = (d_1,
  d_2)$. Therefore $\mathbf{J}=\mathbf{I}_{d_1, d_2}$. We leave the
  remaining details to the reader.
\end{proof}

\begin{theorem}[Inverse Axiom]
  Let $\mathbf{u} \in \mathfrak{G}_{\operatorname{nf}}$. Then $\mathbf{u} \star
  \mathbf{u}' =\mathbf{I}_{\Delta (\mathbf{u}')}$ and $\mathbf{u}'
  \star \mathbf{u}=\mathbf{I}_{\Delta (u)}$ are defined. If $\mathbf{r}
  \star \mathbf{u}$ is defined then $(\mathbf{r} \star \mathbf{u}) \star
  \mathbf{u}' =\mathbf{r}$, while if $\mathbf{u} \star \mathbf{t}$ is
  defined then $\mathbf{u}' \star (\mathbf{u} \star \mathbf{t})
  =\mathbf{t}$.
\end{theorem}

\begin{proof}
  First suppose that $u \in \Omega^{\circ}$. Then $u$ is invertible and
  \[ u' = \gamma (u) u^{- 1}, \qquad \gamma (u) := \frac{\det (B
     (u))}{c_1 (u) c_2 (u)} . \]
  We have (denoting $I_V$)
  \[ \llbracket u, I_{V \otimes V}, u^{- 1} \rrbracket = (u \otimes I) (I
     \otimes I \otimes I) (u^{- 1} \otimes I) - (I \otimes u^{- 1}) (I \otimes
     I \otimes I) (I \otimes u) = 0, \]
  so $u \star u^{- 1}$ is defined and is a constant multiple of $I_{V \otimes
  V}$. Now $u'$ is a constant multiple of $u^{- 1}$, so $u \star u'$ is a
  constant multiple of $I_{V \otimes V}$ and it may be checked that
  $\llbracket u, I, u' \rrbracket = 0$ is the normalized solution, so indeed
	$\mathbf{u} \star \mathbf{u}' = \mathbf{I}_{d_1, d_2}$ where $(d_1, d_2) = \Delta
  (\mathbf{u}') = \Delta (\mathbf{u}^{\ast})$.
  The last assertion follows from associativity and
  Proposition~\ref{prop:idempotents}.

	We have assumed that $u\in\Omega^\circ$. For $\mathbf{u}$ in one
	of the boundary components, By Lemma~\ref{lem:cxdense} we may chose a
	sequence $\{u_n\}\subset\Omega^\circ$ that converges to $\mathbf{u}$ and such
	that $\Delta(u_n)=\Delta(\mathbf{u})$. Then if $\mathbf{r}\star\mathbf{u}$ is
	defined, so is $\mathbf{r}\star u_n$, and we may deduce the general result
  by continuity. 
\end{proof}

From the last three results, we see that $\mathfrak{G}_{\operatorname{nf}}$ is a
groupoid. We will call the disjoint union $\mathfrak{G}_{\operatorname{ff}} \sqcup
\mathfrak{G}_{\operatorname{nf}}$ the {\textit{six-vertex groupoid}}, which accounts for
essentially all Yang-Baxter equations that we may construct from the six-vertex model.

\section{The Five-vertex Groupoid\label{sec:fivever}}
Let us consider the 5-vertex model in which $b_2 (v) = 0$ for all vertices.
Thus $N (v) = a_1 (v) a_2 (v) - c_1 (v) c_2 (v)$. Assuming that $a_1 (v)$ and
$a_2 (v)$ are nonzero, we still may define $v^{\ast}$ with
\begin{equation}
  \label{eq:fivestar} \begin{array}{c}
    a_1 (v^{\ast}) = \frac{c_1 (v) c_2 (v)}{a_1 (v)}, \qquad a_2 (v^{\ast}) =
    \frac{c_1 (v) c_2 (v)}{a_2 (v)},\\
    c_1 (v^{\ast}) = c_2 (v), \qquad c_2 (v^{\ast}) = c_1 (v), \qquad b_1
    (v^{\ast}) = - b_1 (v)
  \end{array}
\end{equation}
and $b_2 (v^{\ast}) = 0$. We have
\[ N (v^{\ast}) = - \frac{c_1 (v) c_2 (v)}{a_1 (v) a_2 (v)} N (v) . \]

There are parametrized Yang-Baxter equation of five-vertex matrices within
the free-fermionic groupoid. But we will ignore thes and instead exhibit a groupoid 
in the non-free-fermionic part, that is disjoint from the six-vertex groupoid
$\mathfrak{G}_{\textrm{nf}}$. Let $\Phi$ be the set of six-vertex
matrices such that $c_1 (v), c_2 (v), N (v), b_1 (v), a_1 (v), a_2 (v)$ are
all nonzero, but $b_2 (v) = 0$. Let $\Phi_b$ be the set such that $c_1 (v),
c_2 (v), a_1 (v), a_2 (v)$ are all nonzero but $b_1 (v) = b_2 (v) = 0$ and $N
(v) = a_1 (v) a_2 (v) - c_1 (v) c_2 (v) = 0$. (The set $\Phi_b$ coincides
with the set $\Omega_b$ in Table~\ref{tab:menagerie}.) Let 
$\overline{\Phi} = \Phi \cup \Phi_b$. In this section we will prove:

\begin{theorem}
  There exists a groupoid $\mathfrak{G}$ with a dominant birational morphism
  $\pi : \mathfrak{G} \rightarrow \overline{\Phi}$ such that $\llbracket \pi
  (u), \pi (u \star v), \pi (v) \rrbracket = 0$ if $u \star v$ is defined.
\end{theorem}

Now $\Delta_1 (v)$ is undefined on the boundary component $\Phi_b$. But on
$\Phi$ we have
\[ \Delta_1 (v) = \frac{N (v)}{a_1 (v) b_1 (v)}, \qquad \Delta_1 (v^{\ast}) =
   \frac{N (v)}{a_2 (v) b_1 (v)} . \]

\begin{proposition}
  Let $u, v \in \Phi$. A necessary and sufficient condition for there to exist
  $w \in S$ such that $\llbracket u, w, v \rrbracket = 0$ is that $\Delta_1
  (u) = \Delta_1 (v^{\ast})$. If this identity is satisfied, then $w$ is
  determined up to a constant multiple, $w \in \overline{\Phi}$ and it may be
  normalized so
  \begin{equation}
    \label{eq:fivemult} \begin{array}{c}
      c_1 (w) = c_1 (u) c_1 (v), \qquad c_2 (w) = c_2 (u) c_2 (v),\\
      a_1 (w) = a_1 (u) a_1 (v), \qquad a_2 (w) = a_2 (u) a_2 (v),\\
      b_1 (w) = b_1 (u) a_1 (v) + a_1 (u^{\ast}) b_1 (v) = a_2 (u) b_1 (v) +
      b_1 (u) a_2 (v^{\ast}) .
    \end{array}
  \end{equation}
  Here the two expressions for $b_1 (w)$ are equivalent on the assumption that
  $\Delta_1 (v) = \Delta_1 (u^{\ast})$. If $w\in\Phi$ then
  \[ \Delta_1 (w) = \Delta_1 (v), \qquad \Delta_1 (w^{\ast}) = \Delta_1
     (u^{\ast}) . \]
\end{proposition}

\begin{proof}
  This is proved similarly to Theorem~\ref{thm:wybe}.
\end{proof}

Now to produce the groupoid let $\mathfrak{G}= \left\{ (v, \varepsilon) \in
\overline{\Phi} \times \mathbb{C}^{\times} |N (v) = a_1 (v) b_1 (v)
\varepsilon \right\}$. Note that the fiber of the projection $\mathfrak{G}
\rightarrow \overline{\Phi}$ onto the first component consists of $(v,
\Delta_1 (v))$ over the open set $\Phi$. On the other hand if $v \in \Phi_b$
the fiber over $v$ if $v \times \mathbb{C}^{\times}$. By abuse of notation if
$v \in \Phi$ then we will identify $v$ with its image $(v, \Delta_1 (v)) \in
\mathfrak{G}$.

Let us introduce the notation $\mathbf{v} = (v, \varepsilon)$ for elements
of $\mathfrak{G}$, where $v \in \overline{\Phi}$ and $\varepsilon \in
\mathbb{C}^{\times}$. Then we extend $\Delta_1$ to all of $\mathfrak{G}$ by
$\Delta_1 (\mathbf{v}) = \varepsilon$. We extend the map $v \mapsto
v^{\ast}$ to all of $\mathfrak{G}$ as follows. If $\mathbf{v}= (v, \Delta_1
(\mathbf{v}))$ with $v \in \Phi_b$ then we define $\mathbf{v}^{\ast} =
(v^{\ast}, \Delta_1 (\mathbf{v}^{\ast}))$ where $v^{\ast}$ is given by the
formula (\ref{eq:fivestar}), and
\[ \Delta_1 (\mathbf{v}^{\ast}) = \frac{a_1 (v)}{a_2 (v)} \Delta
   (\mathbf{v}) . \]

The composition $\mathbf{u} \star \mathbf{v}$ is defined only if $\Delta_1
(\mathbf{u}) = \Delta_1 (\mathbf{v}^{\ast})$. In this case we define
$\mathbf{u} \star \mathbf{v}=\mathbf{w}$ where $\mathbf{w}= (w,
\Delta_1 (\mathbf{w}))$ with $w$ defined by the equations
(\ref{eq:fivemult}) and $\Delta_1 (\mathbf{w}) = \Delta_1 (\mathbf{v})$.

The groupoid also has an inverse map, which we will denote $\mathbf{v}
\mapsto \mathbf{v}'$, in order to avoid confusion with the matrix inverse.
If $v \in \Phi$, then $v^{- 1}, v^{\ast}$ and $v'$ are scalar multiples of one
another. So $\mathbf{v}'$ is not quite the same $\mathbf{v}^{\ast}$ but is
a different normalization: if $\mathbf{v}= (v, \varepsilon)$ then $v' =
\frac{1}{c_1 (v) c_2 (v)} v^{\ast}$ and $\mathbf{v}' = (v', \Delta_1
(\mathbf{v}^{\ast}))$.

The groupoid $\mathfrak{G}$ is topologized so that $\Delta_1$ is continuous.
This means that if $\mathbf{v}= (v, \varepsilon)$ a sequence $\{ v_n \}
\subset \Phi$ converges to $\mathbf{v}$ if $v_n \rightarrow v$ and $\Delta_1
(v_n) \rightarrow \varepsilon$. All of the functions $a_1, a_2, b_1, c_1, c_2,
N, \Delta_1, \mathbf{v}^{\ast}, \mathbf{v}'$ are continuous
on~$\mathfrak{G}$.

\begin{lemma}
  \label{lem:fivecrit}The composition $\mathbf{u} \star \mathbf{v}$ is
  defined if and only if $\Delta_1 (\mathbf{u}) = \Delta_1
  (\mathbf{v}^{\ast})$. If this is so, and $\mathbf{w}=\mathbf{u} \star
  \mathbf{v}$ then $\Delta_1 (\mathbf{w}) = \Delta_1 (\mathbf{v})$ and
  $\Delta_1 (\mathbf{w}^{\ast}) = \Delta_1 (\mathbf{u}^{\ast})$.
\end{lemma}

\begin{proof}
  The first statement is true by definition. For the second statement, these
  identities are true on the dense open subset $\Phi$ by
  Theorem~\ref{thm:wybe} and the general case follows by continuity.
\end{proof}

We note that the operations $\mathbf{v} \mapsto \mathbf{v}^{\ast}$,
$\mathbf{v} \mapsto \mathbf{v}'$ and $(\mathbf{u}, \mathbf{v}) \mapsto
\mathbf{u} \star \mathbf{v}$ are continuous. To confirm that we have a
groupoid we need to prove associativity.

\begin{proposition}
  Let $\mathbf{r}, \mathbf{t}, \mathbf{u} \in \mathfrak{G}$ such that
  $\mathbf{r} \star \mathbf{t}$ and $\mathbf{t} \star \mathbf{u}$ are
  defined. Then $(\mathbf{r} \star \mathbf{t}) \star \mathbf{u}$ and
  $\mathbf{r} \star (\mathbf{t} \star \mathbf{u})$ are both defined, and
  they are equal.
\end{proposition}

\begin{proof}
  this follows formally from Lemma~\ref{lem:fivecrit} that $(\mathbf{r}
  \star \mathbf{t}) \star \mathbf{u}$ and $\mathbf{r} \star
  (\mathbf{t} \star \mathbf{u})$ are defined, because $\Delta_1
  (\mathbf{r} \star \mathbf{t}) = \Delta_1 (\mathbf{t}) = \Delta_1
  (\mathbf{u}^{\ast})$ and $\Delta (\mathbf{r}) = \Delta_1
  (\mathbf{t}^{\ast}) = \Delta_1 ((\mathbf{t} \star
  \mathbf{u})^{\ast})$.
  
  To show that they are equal, we prove this first for $\mathbf{r}= (r,
  \Delta_1 (r))$, $\mathbf{t}= (t, \Delta_1 (t))$ and $\mathbf{u}= (u,
  \Delta_1 (u))$ in the interior, so $r, t, u \in \Phi$. We assume that these
  are in general position so that also $r \star t, t \star u \in \Phi$. To
  show that $(r \star t) \star u = r \star (t \star u)$, the computation in
  Proposition~\ref{prop:gpogenasso} may be adapted. We have easily
  \[ a_1 ((r \star t) \star u) = a_1 (r \star t) a_1 (u) = a_1 (r) a_1 (t) a_1
     (u) = a_1 (r \star (t \star u)), \]
  and similarly for $a_2, c_1$ and $c_2$. For $b_1$:
  \[ \begin{array}{lll}
       b_1 ((r \star t) \star v) & = & a_1 (t^{\ast} \star r^{\ast}) b_1 (v) +
       b_1 (r \star t) a_1 (v)\\
       & = & a_1 (r^{\ast}) a_1 (t^{\ast}) b_1 (v) + (b_1 (r) a_1 (t) + a_1
       (r^{\ast}) b_1 (t)) a_1 (v)\\
       & = & a_1 (r^{\ast}) (a_1 (t^{\ast}) b_1 (v) + b_1 (t) a_1 (v)) + b_1
       (r) a_1 (t) a_1 (v)\\
       & = & a_1 (r^{\ast}) b_1 (t \star v) + b_1 (r) a_1 (t \star v)\\
       & = & b_1 (r \star (t \star v)) .
     \end{array} \]
  These calculations show that $(\mathbf{r} \star \mathbf{t}) \star
  \mathbf{u}=\mathbf{r} \star (\mathbf{t} \star \mathbf{u})$ when
  $\mathbf{r}, \mathbf{t}$ and $\mathbf{u}$ are in a dense open set. The
  general case follows by continuity.
\end{proof}

\section{Solvable Lattice Models\label{sec:solvable}}

Assume we have a groupoid parametrized Yang-Baxter equation $\pi:\mathfrak{G}\to
\operatorname{End}(V\otimes V)$ with object map $\Delta:\mathfrak{G}\to M$.
We will construct
solvable lattice models whose Boltzmann weights come from the
model. We will show that these models have both row and column
solvability. These are familiar results for models associated
with \textit{group} parametrized Yang-Baxter equations.
(Compare~\cite{BBF11,KuperbergASM} for example.)
However it is worth checking that such standard solvability
arguments work for groupoid parametrized Yang-Baxter equations,
and can produce models that are not possible using only a
subgroup of a given groupoid.

The model will be based on an $m\times n$ grid. At the
vertex in the $i$-th row and the $j$-th column
vertex we wish to pick an element $\gamma_{ij}$ such
that the Boltzmann weights at the vertex are the matrix
coefficients of $\pi(\gamma_{ij})$ in the usual way.

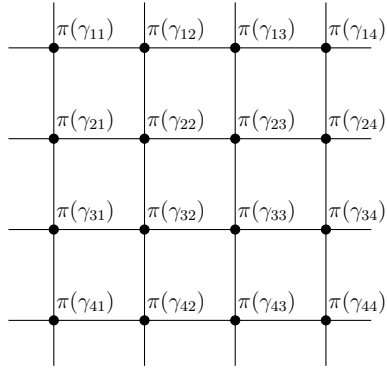
\[
    \begin{tikzpicture}[scale=.6,every node/.style={scale=.7}]
    \foreach \i in {1,3,5,7} {
    \draw (0,\i) to (8,\i);
    \draw (\i,0) to (\i,8);
    \foreach \j in {1,3,5,7}
    \draw[fill=black] (\i,\j) circle (.1);
}
        \foreach \i/\j/\c in {1/1/41,1/3/31,1/5/21,1/7/11,3/1/42,3/3/32,3/5/22,3/7/12,
        5/1/43,5/3/33,5/5/23,5/7/13, 7/1/44,7/3/34,7/5/24,7/7/14}
        \node at (\i.7,\j.4) {$\pi(\gamma_{\c})$};
\end{tikzpicture}
\]

Before we show how that $\gamma_{ij}$ can be produced,
we consider the properties they need for solvability.
Then we will show a systematic procedure for producing
such $\gamma_{ij}$.

We will say that the model is \textit{row solvable} if for every
pair of adjacent rows we may attach an R-matrix $\pi(\rho)$ for some
$\rho\in\mathfrak{G}$ such that we can apply the Yang-Baxter equation.
Let us illustrate this with a $2\times2$ grid. After attaching  the
R-matrix, the grid looks like this:
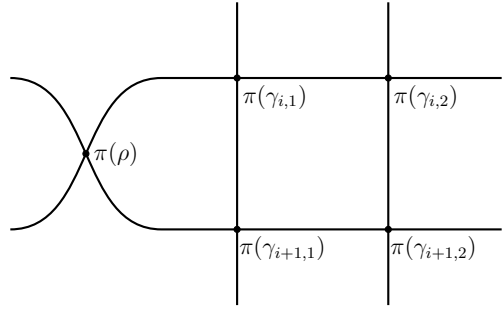
\[
\begin{tikzpicture}[every node/.style={scale=.75}]
\draw[thick] (0,1) \too (2,3) \too (6.5,3);
\draw[thick] (0,3) \too (2,1) \too (6.5,1);
\draw[thick] (3,0) -- (3,4);
\draw[thick] (5,0) -- (5,4);
\foreach \i/\j in {1/2,3/1,3/3,5/1,5/3}
	\draw[fill=black] (\i,\j) circle (.04);
	\node at (3.5,2.75) {$\pi(\gamma_{i,1})$};
	\node at (5.5,2.75) {$\pi(\gamma_{i,2})$};
	\node at (3.6,0.75) {$\pi(\gamma_{i+1,1})$};
	\node at (5.6,0.75) {$\pi(\gamma_{i+1,2})$};
	\node at (1.4,2) {$\pi(\rho)$};
\end{tikzpicture}
\]
After applying the Yang-Baxter equation:
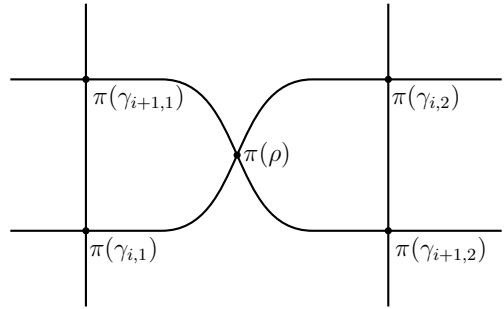
\[
\begin{tikzpicture}[every node/.style={scale=.8}]
\draw[thick] (0,1) \too (2,1) \too (4,3) \too (6.5,3);
\draw[thick] (0,3) \too (2,3) \too (4,1) \too (6.5,1);
\draw[thick] (1,0) -- (1,4);
\draw[thick] (5,0) -- (5,4);
\foreach \i/\j in {3/2,1/1,1/3,5/1,5/3}
	\draw[fill=black] (\i,\j) circle (.04);
	\node at (1.7,2.75) {$\pi(\gamma_{i+1,1})$};
	\node at (5.5,2.75) {$\pi(\gamma_{i,2})$};
	\node at (1.5,0.75) {$\pi(\gamma_{i,1})$};
	\node at (5.7,0.75) {$\pi(\gamma_{i+1,2})$};
	\node at (3.4,2) {$\pi(\rho)$};
\end{tikzpicture}
\]
Here the Yang-Baxter equation is
$\llbracket\pi(\rho),\pi(\gamma_{i,1}),\pi(\gamma_{i+1,1})\rrbracket=0$,
so we need $\rho\star\gamma_{i+1,1}=\gamma_{i,1}$ in $\mathfrak{G}$,
or $\rho=\gamma_{i,1}\star\gamma_{i+1,1}'$ in $\mathfrak{G}$. Then
But we want to use the Yang-Baxter equation again so we need
$\rho=\gamma_{i,2}\star\gamma_{i+1,2}'$. Thus we need
\[\gamma_{i,1}\star\gamma_{i+1,1}'=\gamma_{i,2}\star\gamma_{i+1,2}'\]
in $\mathfrak{G}$. In order to finish the train argument (running the
R-matrix past every column) we therefore need:

\begin{desideratum}[Requirement for row solvability]\label{des:rowsolv}
  For every $i,i+1$ we need $\gamma_{i,j}\star\gamma_{i+1,j}'$ to
  be defined and independent of~$j$.
\end{desideratum}

Column solvability would be handled the same way, attaching an R-matrix
above the $j,j+1$ columns. Similarly:

\begin{desideratum}[Requirement for column solvability]\label{des:colsolv}
  For every $j,j+1$ we need $\gamma_{i,j}'\star\gamma_{i,j+1}$ to
  be defined and independent of~$i$.
\end{desideratum}

\begin{proposition}
  Row solvability is equivalent to column solvability.
\end{proposition}

\begin{proof}
  Assume row solvability, so $\gamma_{i, j} \star \gamma_{i + 1, j}'$ is
  defined and $\gamma_{i, j} \star \gamma'_{i + 1, j} = \gamma_{i, j + 1}
  \star \gamma_{i + 1, j + 1}'$. By the inverse groupoid axiom $\gamma_{i, j}'
  \star (\gamma_{i, j} \star \gamma_{i + 1, j}')$ is defined and equals
  $\gamma_{i + 1, j}'$. Therefore $\gamma_{i + 1, j}' = \gamma_{i, j}' \star
  \gamma_{i, j + 1} \star \gamma_{i + 1, j + 1}'$. Another application of the
  inverse axiom gives $\gamma_{i + 1, j}' \star \gamma_{i + 1, j + 1} =
  \gamma_{i, j}' \star \gamma_{i + 1, j}$.
\end{proof}

Now let us show how to construct models satisfying these desiderata. We fix an
element $d \in \mathfrak{G}$, and choose $\phi_1, \cdots, \phi_n \in
\mathfrak{G}$ such that $\Delta (\phi_i) = d$, and $\psi_1, \cdots, \psi_m$
such that $\Delta (\psi_j') = d$. Then $\Delta (\phi_i) = \Delta (\psi_j')$ so
$\gamma_{i, j} := \phi_i \star \psi_j$ is defined for all $i, j$.

\begin{proposition}
  With this construction $\gamma_{i, j} \star \gamma_{i + 1, j}'$ is defined
  and independent of $j$. Hence the row and column solvability desiderata are
  satisfied and the model is solvable.
\end{proposition}

\begin{proof}
  Let us show that $\gamma_{i, j} \star \gamma_{i + 1, j}'$ is defined and
  equals $\phi_i \star \phi_{i + 1}'$. Indeed, $\Delta (\gamma_{i, j}) =
  \Delta (\phi_i \star \psi_j) = \Delta (\psi_j) = \Delta (\gamma_{i + 1,
  j})$. This implies that $\gamma_{i, j} \star \gamma_{i + 1, j}'$ is defined.
  It equals
  \[ \phi_i \star \psi_j \star (\phi_{i + 1} \star \gamma_j)' = \phi_i \star
     \psi_j \star \psi_j' \star \phi_{i + 1}' = \phi_i \star \phi_{i + 1}' .
  \]
  This is independent of $j$ proving row solvability. 
\end{proof}

We have therefore constructed a class of solvable lattice models using a
groupoid parametrized Yang-Baxter equation. We will now address the question
of whether these are truly more general than models that could be constructed
using group parametrized Yang-Baxter equations. After all, we have imposed a
condition that $\Delta (\phi_i) = \Delta (\psi_j') = d$ for some fixed element
$d \in \mathfrak{G}$. How strongly does this constrain the $\gamma_{i, j}$?

Before addressing this question, we mention a criterion for a groupoid to be a
disjoint union of groups. If $\mathfrak{G}$ is a disjoint union of groups, and
$v \in \mathfrak{G}$ then $v \star v$ is always defined, and so $\Delta (v) =
\Delta (v')$. From this we see that if $\llbracket u,v,w\rrbracket=0$ we have
$\Delta(u)=\Delta(v)=\Delta(w)$.

On the other hand, in the six- and five-vertex groupoids, $\Delta (v)$ and
$\Delta (v')$ are decoupled. That is, if we define $M_d = \{ \Delta (v') |v
\in M, \Delta (v) = d \}$, the set $M_d$ is large.

Now consider that
\[ \Delta (\gamma_{i j}) = \Delta (\psi_j), \qquad \Delta (\gamma_{i j}') =
   \Delta (\phi_i') . \]
Although $\Delta (\phi_i) = \Delta (\psi_j') = d$ for all $i, j$, the values
of $\Delta (\psi_j)$ and $\Delta (\phi_i')$ can take arbitrary values from
$M_d$. Thus $\Delta (\gamma_{i j})$ and $\Delta (\gamma_{i j}')$ can also be
made arbitrary elements of $M_d$, subject only to the constraint that $\Delta
(\gamma_{i j})$ depends only on the column $j$, while $\Delta (\gamma_{i j}')$
depends only on the row~$i$. This is in contrast with classical models such as
those in {\cite[Chapter 9]{BaxterBook}}, {\cite{KuperbergASM}} or
{\cite{BBF11}}, where $\Delta (\gamma_{i j})$ and $\Delta (\gamma_{i j}')$ are
necessarily the same for all~$i, j$.

\bibliographystyle{hplain}
\bibliography{groupoid}

\end{document}